\documentclass[11pt, reqno]{amsart}
\usepackage{amscd}        % Package used to produce simple commutative diagrams

\usepackage[english]{babel}

\usepackage{booktabs}
\usepackage[svgnames, dvipsnames, table]{xcolor}
\usepackage[colorlinks = true, allcolors = cyan, pagebackref=true, colorlinks, pdfencoding=auto]{hyperref}

\usepackage{tabu}

\usepackage{siunitx} %To format numbers with the command \number{}

\usepackage{orcidlink}

\usepackage{pdflscape}

\usepackage{float}

\usepackage{nomencl}

\usepackage{algorithm}
\usepackage{algpseudocode}

\usepackage{multirow}

\usepackage{tikz-cd}
\usepackage{graphicx}
\usepackage{rotating}
\usepackage{diagbox}
\usepackage{amssymb}
\usepackage{epstopdf}
\usepackage{tikz}
\definecolor{UCRceleste}{RGB}{0,192,243}
\definecolor{Crimson}{RGB}{220, 20, 60}
\definecolor{Turquoise}{HTML}{04d4f0}
\definecolor{GunmetalGray}{HTML}{1ca6a0}
\definecolor{HotPink}{HTML}{f652a0}
\definecolor{BlueGrotto}{HTML}{059dc0}
\definecolor{GaloisBlue}{HTML}{6495ED}
\definecolor{GaloisRed}{HTML}{DE3163}
\definecolor{mintgreen}{RGB}{152,255,152}
\definecolor{pinksalmon}{RGB}{255,102,102}
\definecolor{hueso}{RGB}{245,245,220}
\definecolor{marfil}{RGB}{255,253,208}
\definecolor{amarillo}{RGB}{255,255,0}
\usetikzlibrary{decorations.markings,arrows}
\usetikzlibrary{decorations.pathreplacing}
\usepackage[inner=1.0in,outer=1.0in,bottom=1.0in, top=1.0in]{geometry}

\numberwithin{equation}{section}
\newtheorem{theorem}{Theorem}[section]
\newtheorem{lemma}[theorem]{Lemma}
\newtheorem{proposition}[theorem]{Proposition}

\makeatletter
\def\moverlay{\mathpalette\mov@rlay}
\def\mov@rlay#1#2{\leavevmode\vtop{%
   \baselineskip\z@skip \lineskiplimit-\maxdimen
   \ialign{\hfil$\m@th#1##$\hfil\cr#2\crcr}}}
\newcommand{\charfusion}[3][\mathord]{
    #1{\ifx#1\mathop\vphantom{#2}\fi
        \mathpalette\mov@rlay{#2\cr#3}
      }
    \ifx#1\mathop\expandafter\displaylimits\fi}
\makeatother

\newcommand{\suchthat}{\;\ifnum\currentgrouptype=16 \middle\fi|\;}

\newcommand{\Z}{\mathbb{Z}}

\newcommand{\Q}{\mathbb{Q}}
\newcommand{\R}{\mathbb{R}}

\newcommand{\ord}[1]{\operatorname{ord}#1}

\theoremstyle{definition}
\newtheorem{remark}[theorem]{Remark}

\def\R{{\mathbb R}}
\def\Z{{\mathbb Z}}
\def\Q{{\mathbb Q}}

\def\O_K{{\Cal{O}_{K}}}
\def\O_F{{\Cal{O}_{F}}}
\def\N_F{{\Cal{N}_{F/\Q}}}

\usepackage{listings}
\definecolor{sagebrown}{RGB}{176, 92, 10}
\definecolor{sageblue}{RGB}{44, 45, 254}
\definecolor{sagepurple}{RGB}{151, 57, 164}
\definecolor{sagegreen}{RGB}{18, 103, 68}
\definecolor{sagered}{RGB}{170, 16, 15}
\lstdefinestyle{SageMath-style}{
    backgroundcolor=\color{white},   
    commentstyle=\color{sagebrown},
    keywordstyle=\color{sagepurple},
    keywordstyle = [2]{\color{sageblue}},
    keywordstyle = [3]{\color{yellow}},
    numberstyle=\tiny\color{sagegreen},
    stringstyle=\color{sagered},
    basicstyle=\ttfamily\footnotesize,
    breakatwhitespace=false,         
    breaklines=true,                 
    captionpos=b,                    
    keepspaces=true,                 
    numbers=left,                    
    numbersep=5pt,                  
    showspaces=false,                
    showstringspaces=false,
    showtabs=false,                  
    tabsize=2
}

\begin{document}

\title{Counting elliptic curves over $\Q$ with bounded naive height}

\author[]{Adrian Barquero-Sanchez\orcidlink{0000-0001-7847-2938} and Daniel Mora-Mora\orcidlink{0009-0000-9302-0497}}

\address{Escuela de Matem\'atica, Universidad de Costa Rica, San Jos\'e 11501, Costa Rica}
\address{Centro de Investigación en Matemática Pura y Aplicada, Universidad de Costa Rica, San Jos\'e 2060, Costa Rica}
\email{adrian.barquero\_s@ucr.ac.cr}
\email{daniel.moramora@ucr.ac.cr}

\subjclass{11G05, 11G15, 11N45}
\keywords{Elliptic curve, counting, asymptotic formula, naive height, j-invariant, complex multiplication}

\begin{abstract}

In this paper, we give exact and asymptotic formulas for counting elliptic curves $ E_{A,B} \colon y^2 = x^3 + Ax + B $ with $ A, B \in \mathbb{Z} $, ordered by naive height. We study the family of all such curves and also several natural subfamilies, including those with fixed $j$-invariant and those with complex multiplication (CM). In particular, we provide formulas for two commonly used normalizations of the naive height appearing in the literature: the \emph{calibrated naive height}, defined by
\[
H^{\mathrm{cal}}(E_{A,B}) := \max\{ 4|A|^3, 27B^2 \},
\]
and the \emph{uncalibrated naive height}, defined by
\[
H^{\mathrm{ncal}}(E_{A,B}) := \max\{ |A|^3, B^2 \}.
\]
In fact, we prove our theorems with respect to the more general naive height $H_{\alpha, \beta}(E_{A,B}) := \max\{ \alpha |A|^3, \beta B^2 \}$, defined for arbitrary real parameters $\alpha,\beta \geq 1$.

As part of our approach, we give a completely explicit parametrization of the curves $E_{A,B}$ with fixed $j$-invariant and bounded naive height, including explicit equations for the minimal-height models and a description of all curves in the family as twists of them. We also include tables comparing and verifying our theoretical predictions with exact counts obtained via exhaustive computer searches, and we compute data for CM elliptic curves of naive height up to $10^{30}$. Code in \texttt{SageMath} is provided to compute all exact and asymptotic formulas appearing in the paper.

\end{abstract}

\maketitle

\section{Introduction} 

In recent decades, arithmetic statistics for elliptic curves has been a very active area of research in number theory. Although there are various ways to order elliptic curves when studying their distribution with respect to a given property—for example, by the size of their conductor or their Faltings height—it has proven especially fruitful, both theoretically and computationally, to order them using simple height functions defined directly in terms of their defining coefficients. In particular, in this paper we will use the notion of the \textit{naive height}, which measures the curves by the size of their coefficients. More precisely, for an elliptic curve in short Weierstrass form $E_{A, B} \colon y^2 = x^3 + Ax + B$, with $A, B \in \Z$, we will consider two common normalizations of this height that are widely used in the literature, which, following the terminology introduced by Balakrishnan et al. in \cite{Bal16} in order to differentiate them, we will call the \textit{(calibrated) naive height}, defined by
\begin{align*}
H^{\mathrm{cal}}(E_{A,B}) := \max{\{ 4|A|^3, 27B^2\}}
\end{align*}
and the \textit{uncalibrated naive height}, defined by
\begin{align*}
H^{\mathrm{ncal}}(E_{A,B}) := \max{\{ |A|^3, B^2\}}.
\end{align*}

Over the past three decades, several important works in the arithmetic statistics of elliptic curves have relied on these normalizations of the naive height. Notably, Bhargava and Shankar studied the distribution of ranks \cite{BS15a, BS15b}, while Harron and Snowden \cite{HS17}, as well as Cullinan, Kenney, and Voight \cite{CKV22}, examined the distribution of curves with prescribed torsion subgroups. More recently, the problem of counting elliptic curves admitting an $m$-isogeny for various $m$ has attracted attention (see e.g. \cite{PPV20}, \cite{MV23}, \cite{BS24}, \cite{santiago25}). Another direction, pursued by Fouvry, Nair, and Tenenbaum \cite{FNT92}, investigates the density of curves $E_{A,B}$ in connection with the Szpiro conjecture.

Exploring a different direction, the first author and Calvo-Monge recently studied the distribution of CM elliptic curves over $\Q$ when ordered by calibrated naive height \cite{BSCM25}. Recall that an elliptic curve $E/\Q$ is said to have \textit{(potential) complex multiplication} if its endomorphism ring $\operatorname{End}_{\overline{\Q}}(E)$ is larger than $\Z$. It is known that in that case we must have
$$
\operatorname{End}_{\overline{\Q}}(E) \cong \mathbb{Z}+f \mathcal{O}_K,
$$
where $\mathcal{O}_K$ is the ring of integers in an imaginary quadratic field $K = \Q(\sqrt{d_K})$ of discriminant $d_K < 0$ and $f \geq 1$ is an integer, and the theory of complex multiplication determines that there are exactly thirteen possibilities for the pair $(d_K, f)$, which are given in Table \ref{tab:cm-pairs}. A concise overview of this is provided in \cite[Section 1.2]{BSCM25}.

\begin{table}[H]
\centering
\begin{tabular}{|c|c|c|c|c|c|c|c|c|c|c|c|c|c|}
\hline
$d_K$ & $-3$ & $-3$ & $-3$ & $-4$ & $-4$ & $-7$ & $-7$ & $-8$ & $-11$ & $-19$ & $-43$ & $-67$ & $-163$ \\
\hline
$f$ & $1$  & $2$  & $3$  & $1$  & $2$  & $1$  & $2$  & $1$  & $1$   & $1$   & $1$   & $1$   & $1$    \\
\hline
\end{tabular}
\caption{The thirteen possible values of the pair $(d_K, f)$ for CM elliptic curves over $\Q$.}
\label{tab:cm-pairs}
\end{table}

Now, since the $j$-invariant of an elliptic curve over $\Q$ completely determines its $\overline{\Q}$-isomorphism class—and that determines its endomorphism ring—, in order to study the distribution of CM elliptic curves over $\Q$, the first author and Calvo-Monge tackled in \cite{BSCM25} the more general problem of how elliptic curves with a fixed $j$-invariant are distributed.

In particular, in this paper we make significant improvements to some of the results of \cite{BSCM25}. More specifically, in that paper the authors proved asymptotic formulas for the number of $\Q$-isomorphism classes of elliptic curves of the form $E_{A,B}$ with $A, B \in \Z$, having a fixed $j$-invariant and calibrated naive height bounded by $X$ (see \cite[Theorem 1.6]{BSCM25}). The count was carried out by enumerating those curves $E_{A,B}$ satisfying the height bound and having the prescribed $j$-invariant. For example, they showed that for $j$-invariant 0, the number of such curves with calibrated height up to $X$ is of the form
$$
\frac{2}{3^{3 / 2} \zeta(6)} X^{1/2} + O(X^{1/12})
$$
and for the curves with $j$-invariant $1728$ the number of curves satisfies the asymptotic formula
$$
\frac{2^{1 / 3}}{\zeta(4)} X^{1/3} + O(X^{1/12}).
$$
However, when $j \in \Q \smallsetminus \{ 0, 1728 \}$, in that paper it was only possible to prove that this number was $O(X^{1/6})$. In this paper we improve that count and show that it is actually of the form
$$
\frac{C_j}{\zeta(2)} X^{1/6} + O(X^{1/12})
$$
for a completely explicit and easily computable constant $C_j$ that only depends on the value of $j \in \Q \smallsetminus \{ 0, 1728 \}$ and on the chosen normalization of the naive height (see Theorem \ref{thm:fixed-j-theorem-generalized},  Table \ref{table:main-terms-calibrated-uncalibrated} and equation \eqref{eqn:cj-constant-definition}). 

Moreover, as a key ingredient to our improvements, we have found a completely explicit way to parametrize the family of elliptic curves $E_{A, B}$ having $j(E_{A, B}) = j$, in terms of an integer parameter. In particular, our formula gives the two curves of the form $E_{A, B}$ of minimal naive height for any given $j$-invariant and the parametrization gives all the curves with fixed $j$-invariant essentially as the set of quadratic twists of these curves of minimal naive height (see Theorem \ref{thm:parametrization-and-exact-count-curves-with-j-introduction} and Remark \ref{remark:parametrization-remark}). This parametrization also allows us to prove \textit{exact formulas} for the number of all elliptic curves of the form $E_{A, B}$ (see Theorem \ref{thm:global-counts-generalized-intro}), for those with fixed $j$-invariant (see Theorem \ref{thm:fixed-j-theorem-generalized}) and for those with complex multiplication (see Theorem \ref{thm:CM-theorem-generalized}), when ordered by calibrated or uncalibrated naive height.

For example, to illustrate the results mentioned in the previous paragraph, if $E/\Q$ is a CM elliptic curve with $\operatorname{End}_{\overline{\Q}}(E) \cong \Z\left[\frac{1+\sqrt{-163}}{2}\right]$, the corresponding $j$-invariant is $j=-2^{18} \cdot 3^3 \cdot 5^3 \cdot 23^3 \cdot 29^3$ and in that particular case, Theorem \ref{thm:parametrization-and-exact-count-curves-with-j-introduction} implies that every elliptic curve $E_{A, B} \colon y^2 = x^3 + Ax + B$ having $\operatorname{End}_{\overline{\Q}}(E_{A, B})  \cong \Z\left[\frac{1+\sqrt{-163}}{2}\right]$ is contained in the parametrized family of curves
\begin{align}\label{eqn:curves-163-parametrization}
E_{A(m), B(m)} \colon y^2 = x^3  + A(m)x + B(m),
\end{align}
where $m \in \Z \smallsetminus \{ 0 \}$ and the coefficients $A(m)$ and $B(m)$ are given by
\begin{align*}
A(m) := -(2^{4} \cdot 5 \cdot 23 \cdot 29 \cdot 163)m^2 \quad \text{and} \quad
B(m) := -(2 \cdot 7 \cdot 11 \cdot 19 \cdot 127 \cdot 163^{2})m^3.
\end{align*}
The two curves of minimal (calibrated) naive height with complex multiplication by this order are the curves given by the values $m = \pm 1$, that is, the curves
\begin{align}\label{eqn:curves-of-minimal-height-163}
y^2=x^3 - 8697680x \pm 9873093538.
\end{align}
It is interesting to mention that in \cite[Table 4]{BSCM25}, the authors show the results of a computer search for all elliptic curves $E_{A, B}$ with complex multiplication and calibrated naive height at most $10^{10}$, in which no curves were found with CM by the order $\Z\left[\frac{1+\sqrt{-163}}{2}\right]$. Our current results—which were found theoretically and not as the product of a computer search—explain this fact, since the curves \eqref{eqn:curves-of-minimal-height-163} of minimal calibrated naive height in that case have height
$$
\num{2631905352272628650988} \approx 2.6319\times 10^{21}.
$$
It is also noteworthy that, at the time of writing this\footnote{June 2025}, the curves \eqref{eqn:curves-of-minimal-height-163} do not appear in the LMFDB, and in fact, there are only six elliptic curves with CM by the order $\Z\left[\frac{1+\sqrt{-163}}{2}\right]$ in the database, only two of which are in short Weierstrass form; explicitly, they are the elliptic curve \href{https://www.lmfdb.org/EllipticCurve/Q/425104/g/1}{425104.g1} and the elliptic curve \href{https://www.lmfdb.org/EllipticCurve/Q/425104/g/2}{425104.g2}, which correspond to the curves given by the values $m=2$ and $m=-326$ in the parametrization \eqref{eqn:curves-163-parametrization}, respectively.\footnote{In the LMFDB, the elliptic curves are currently ordered by conductor, and out of the six curves with CM by the order $\Z\left[\frac{1+\sqrt{-163}}{2}\right]$ that appear in the database, the ones with the largest conductor are precisely the two curves \cite[\href{https://www.lmfdb.org/EllipticCurve/Q/425104/g/1}{Elliptic Curve 425104.g1}]{lmfdb} and \cite[\href{https://www.lmfdb.org/EllipticCurve/Q/425104/g/2}{Elliptic Curve 425104.g2}]{lmfdb}, which have conductor $2^4 \cdot 163^2 = \num{425104}$, whereas the curves \eqref{eqn:curves-of-minimal-height-163} of minimal calibrated naive height having CM by this order have conductor $2^6 \cdot 163^2 = \num{1700416}$.} 

As a complement to our theoretical results, in Section \ref{section:explicit-computations} we include explicit computations and data for CM elliptic curves ordered by naive height. These extend to very large height bounds and provide both curve counts and minimal-height representatives for each of the thirteen CM $j$-invariants.

\section{Definitions and Statement of Results}\label{section: definitions-and-results}

\subsection{Families of elliptic curves}

In this paper, we will study two widely used families of elliptic curves, which we now define. For each pair $(A, B) \in \Z^2$ we define the curve $E_{A, B} \colon y^2 = x^3 + Ax + B$ and first consider the family
\begin{align}\label{def:E-tilde-family}
\widetilde{\mathcal{E}} := \{ E_{A, B} \suchthat \text{$A, B \in \Z$ and $\Delta_{E_{A, B}} := -16(4A^3 + 27B^2) \neq 0$}  \}
\end{align}
consisting of all elliptic curves in short Weierstrass form with coefficients in $\Z$. It is known that each elliptic curve $E/\Q$ is $\Q$-isomorphic to a unique elliptic curve $E_{A, B} \in \widetilde{\mathcal{E}}$ such that there is no prime number $p$ with $p^4 \mid A$ and $p^6 \mid B$. Thus, we also consider the family
\begin{align}\label{def:E-family}
\mathcal{E} := \{ E_{A, B} \in \widetilde{\mathcal{E}} \suchthat \text{there is no prime number $p$ with $p^4 \mid A$ and $p^6 \mid B$}  \}.
\end{align}
This subfamily of $\widetilde{\mathcal{E}}$ is a complete set of representatives for the $\Q$-isomorphism classes of 
elliptic curves $E/\Q$.

\subsection{Fixed \texorpdfstring{$j$}{j}-invariant and CM subfamilies}

Next, we will also consider the following subfamilies of $\widetilde{\mathcal{E}}$ and $\mathcal{E}$. Recall first that the $j$-invariant of an elliptic curve $E_{A, B}$ is given by
$$
j(E_{A, B}) = 1728 \frac{4A^3}{4A^3 + 27B^2}.
$$
Then, we consider the subfamilies of $\widetilde{\mathcal{E}}$ and $\mathcal{E}$ obtained by restricting only to elliptic curves with a fixed $j$-invariant. Thus, for each $j \in \Q$, define
\begin{align}\label{def: elliptic-j-invariant-family}
\widetilde{\mathcal{E}}_j := \{ E_{A, B} \in \widetilde{\mathcal{E}} \suchthat j(E_{A, B}) = j \} \quad \text{and} \quad \mathcal{E}_j := \{ E_{A, B} \in \mathcal{E} \suchthat j(E_{A, B}) = j \}.
\end{align}

Similarly, we define the subfamilies of $\widetilde{\mathcal{E}}$ and $\mathcal{E}$ consisting of the elliptic curves that have \textit{(potential) complex multiplication}, i.e., their endomorphism ring $\operatorname{End}_{\overline{\Q}}(E_{A, B})$ is larger than $\Z$. Thus, we define
\begin{align}
\widetilde{\mathcal{E}}^{\mathrm{cm}} := \{ E_{A, B} \in \widetilde{\mathcal{E}} \suchthat \text{$E_{A, B}$ has potential CM} \} \end{align}
and
\begin{align}
\mathcal{E}^{\mathrm{cm}} := \{ E_{A, B} \in \mathcal{E} \suchthat \text{$E_{A, B}$ has potential CM} \}.
\end{align}

It is known that the elliptic curves over $\Q$ that have potential complex multiplication are completely determined by their $j$-invariant. In fact, there are exactly thirteen $j$-invariants that correspond to elliptic curves over $\Q$ having potential complex multiplication, and they are the elements of the set (see e.g. \cite[Table 3]{BSCM25} and \cite[p. 483]{Sil94})
\begin{align*}
\mathcal{J}^{\textrm{cm}} = \{
& 0, 1728, -3375, 8000, -32768, 54000, 287496, -12288000,\\
& 16581375, -884736, -884736000, -147197952000, -262537412640768000\}.
\end{align*}
This implies that we can decompose the families $\widetilde{\mathcal{E}}^{\textrm{cm}}$ and $\mathcal{E}^{\textrm{cm}}$ as disjoint unions over the thirteen CM $j$-invariants in $\mathcal{J}^{\mathrm{cm}}$, namely
\begin{align}\label{eq:Ecm-decomposition}
    \widetilde{\mathcal{E}}^{\mathrm{cm}} = \coprod_{j \in \mathcal{J}^{\textrm{cm}}} \widetilde{\mathcal{E}}_j \quad \text{and} \quad \mathcal{E}^{\mathrm{cm}} = \coprod_{j \in \mathcal{J}^{\textrm{cm}}} \mathcal{E}_j.
\end{align}

\subsection{Naive heights and finite subfamilies}

In order to study how the curves in the families that we defined above are distributed, we will use the notion of the naive height, which measures the curves by the size of their coefficients. In particular, as mentioned in the introduction, we will consider the \textit{(calibrated) naive height}, defined by
\begin{align}
H^{\mathrm{cal}}(E_{A,B}) := \max{\{ 4|A|^3, 27B^2\}}
\end{align}
and the \textit{uncalibrated naive height}, defined by
\begin{align}
H^{\mathrm{ncal}}(E_{A,B}) := \max{\{ |A|^3, B^2\}}.
\end{align}
We note that the discriminant of a curve $E_{A, B}$ is bounded in terms of these heights by
$$
|\Delta_{E_{A, B}}| \leq 32 H^{\mathrm{cal}}(E_{A, B}) \quad \text{and} \quad |\Delta_{E_{A, B}}| \leq 496 H^{\mathrm{ncal}}(E_{A, B}).
$$

The distinction of a calibrated and an uncalibrated naive height reflects the historical evolution of the asymptotic theory of elliptic curves. Initially, the uncalibrated naive height was utilized by Brumer \cite{Bru92} to systematically enumerate elliptic curves and study the distribution of their ranks. Later, the groundbreaking work of Bhargava and Shankar on the geometry of numbers of binary quartic forms \cite{BS15b}, showed that the calibrated height is better adapted to the invariant theory of binary quartic forms (see e.g. \cite[Section 3.4]{Poo13}). Consequently, modern arithmetic statistics has almost exclusively adopted the calibrated naive height.

 However, as we later show, both naive heights should produce equal asymptotics up to a constant factor. To unify our arguments and to provide more general results that apply both to the calibrated and the uncalibrated heights, we define the \textit{generalized naive height}, defined for any $\alpha,\beta \in \R_{\geq 1}$ by
\begin{align}
H_{\alpha,\beta}(E_{A,B}):=\max{\{\alpha |A|^3, \beta B^2\}}.
\end{align}
In terms of this height we have $H^{\mathrm{cal}} = H_{4, 27}$ and $H^{\mathrm{ncal}} = H_{1, 1}$.

For each of the infinite families of elliptic curves described above, we define corresponding finite subfamilies using the naive height functions introduced earlier. More precisely, given \( X > 0 \), a height function \( H \in \{ H^{\mathrm{cal}}, H^{\mathrm{ncal}}, H_{\alpha,\beta} \} \), and a family \( \mathcal{F} \in \{ \widetilde{\mathcal{E}}, \mathcal{E}, \widetilde{\mathcal{E}}_j, \mathcal{E}_j, \widetilde{\mathcal{E}}^{\mathrm{cm}}, \mathcal{E}^{\mathrm{cm}} \} \), we define the \textit{finite} subfamily
\[
\mathcal{F}(X; H) := \{ E_{A, B} \in \mathcal{F} \suchthat H(E_{A, B}) \leq X \}.
\]

\subsection{Auxiliary constants}

In order to state our results, we need to define the rational numbers that appear in equations \eqref{eqn:Na-definition} and \eqref{eqn:aj-definition} below. For $a \in \Q^{\times}$, let 
$$
a=(-1)^{\varepsilon} \prod_{p \in \mathbb{P}} p^{\operatorname{ord}_p(a)}
$$
be the prime factorization of $a$, where $\mathbb{P} = \{ 2, 3, 5, 7, 11, \dots \}$ denotes the set of all prime numbers. Also, for each $p \in \mathbb{P}$, we define the integer
\[
\alpha_{p}(a):= \begin{cases} 
\left \lceil\dfrac{\operatorname{ord}_{p}(a)}{2} \right \rceil & \text{if $\operatorname{ord}_{p}(a) \geq 0$,} \vspace{5pt}\\ \left \lceil\dfrac{\operatorname{ord}_{p}(a)}{3} \right \rceil & \text{if $\operatorname{ord}_{p}(a) < 0$.}
\end{cases}
\]
Then, using this we define
\begin{align}\label{eqn:Na-definition}
N_a:=\prod_{p \in \mathbb{P}} p^{\alpha_p(a)} \in \mathbb{Q}^{\times}.
\end{align}
Also, for each $j \in \Q \smallsetminus \{ 0, 1728 \}$, we let
\begin{align}\label{eqn:aj-definition}
a(j) := \frac{4(1728 - j)}{27j} \in \Q.
\end{align}
Finally, in several of our formulas we will need the constant
\begin{align}\label{eqn:cj-constant-definition}
c(j;H_{\alpha,\beta}) := \min \left\{\frac{|a(j)|^{1 / 2}}{N_{a(j)} \alpha^{1 / 6}}, \frac{|a(j)|^{1 / 3}}{N_{a(j)} \beta^{1 / 6}}\right\}.
\end{align}
When considering the calibrated and the uncalibrated naive height we will use the notations $c(j;H^{\mathrm{cal}}):=c(j;H_{4,27})$ and $c(j; H^{\mathrm{ncal}}):=c(j; H_{1,1})$, respectively.

We finally remark that in the formulas appearing in the theorems of this paper, the quantity $\zeta(s)$ denotes the value of the Riemann Zeta function
$$
\zeta(s) := \sum_{n = 1}^{\infty} \frac{1}{n^s}
$$
at $s > 1$.

\subsection{Statements of main results}

In this section, we state our exact and asymptotic formulas for the numbers of elliptic curves in the different families considered, when ordered by the generalized naive height $H_{\alpha, \beta}$ for any $\alpha, \beta \in \R_{\geq 1}$.

We have used \texttt{SageMath} \cite{SageMathCoCalc} to verify all the exact formulas stated in the theorems of this paper by comparing them with results obtained through exhaustive computer searches for each of the families considered. In all cases, the computed values agree precisely with those predicted by our formulas. Similarly, we have checked the accuracy of the asymptotic formulas by comparing the predicted approximations with the exact values. The corresponding code is available in the \texttt{GitHub} repository \cite{BSMM25}. 

\subsubsection{Global counting formulas}

We begin with exact and asymptotic formulas for the number of curves in the global families 
\( \widetilde{\mathcal{E}} \) and \( \mathcal{E} \) ordered by naive height. 
The following theorem is stated for the generalized naive height \(H_{\alpha,\beta}\), 
and in particular it recovers previously known results for the calibrated and 
uncalibrated heights.

The case of the uncalibrated naive height \(H^{\mathrm{ncal}}\) was obtained by 
Brumer~\cite[Lemma 4.3]{Bru92}, who proved an asymptotic formula for 
\(\#\mathcal{E}(X;H^{\mathrm{ncal}})\).
A corresponding result for the calibrated naive height \(H^{\mathrm{cal}}\) 
was stated by Poonen in~\cite[Proposition 2.1]{Poo18}. His statement differs slightly 
in that the error term is given as \(o(X^{5/6})\).
A detailed proof of this calibrated case appears in~\cite[Theorem 2.1]{BSCM25}. 
Formula \eqref{eqn:E-global-formula-generalized-intro} in Theorem \ref{thm:global-counts-generalized-intro} unifies these results and extends them to the generalized height 
\(H_{\alpha,\beta}\).

\begin{theorem}\label{thm:global-counts-generalized-intro}
For every $X > 1$, the numbers of elliptic curves in the families $\mathcal{E}(X; H_{\alpha, \beta})$ and $\widetilde{\mathcal{E}}(X; H_{\alpha, \beta})$ satisfy the formulas
\begin{align}\label{eqn:E-global-formula-generalized-intro}
\#\mathcal{E}(X; H_{\alpha, \beta}) = \frac{4}{\alpha^{1/3} \beta^{1/2} \zeta(10)} X^{5/6}+O(X^{1/2})
\end{align}
and
\begin{align}\label{eqn:widetilde-E-global-formula-generalized}
\#\widetilde{\mathcal{E}}(X; H_{\alpha, \beta}) &= \left( 2 \left\lfloor \frac{X^{1/3}}{\alpha^{1/3}} \right\rfloor + 1 \right) \left( 2 \left\lfloor \frac{X^{1/2}}{\beta^{1/2}} \right\rfloor + 1 \right) - 2 \left\lfloor c(\alpha, \beta) X^{1/6} \right\rfloor - 1 \\
&= \frac{4}{\alpha^{1/3} \beta^{1/2}}X^{5/6}+O(X^{1/2}),
\end{align}
where 
\begin{align}\label{eqn:c-alpha-beta-constant-definition}
c(\alpha, \beta) := \min{ \left\{  \frac{1}{3^{1/2} \alpha^{1/6}}, \frac{1}{2^{1/3} \beta^{1/6}} \right\} }.
\end{align}
\end{theorem}

\begin{remark}
In \cite[p.~352]{Bal16}, Balakrishnan et al.\ state that their project produced an exhaustive database of isomorphism classes of elliptic curves \( E_{A, B} \) with (calibrated) naive height up to \( 2.7 \cdot 10^{10} \), yielding a total of \num{238764310} curves. In our notation, this means that they computed the exact value \( \#\mathcal{E}(2.7 \cdot 10^{10}; H^{\mathrm{cal}}) = \num{238764310} \). 

Choosing $(\alpha, \beta) = (4, 27)$, the approximation given by the main term in equation~\eqref{eqn:E-global-formula-generalized-intro} yields the value
\[
\#\mathcal{E}(2.7 \cdot 10^{10}; H^{\mathrm{cal}}) \approx \frac{2^{4/3}}{3^{3/2} \zeta(10)} (2.7 \cdot 10^{10})^{5/6} \approx \num{238815691.23},
\]
which has a relative error of only about \( 0.02\% \).
\end{remark}

\subsubsection{Counting formulas for curves with fixed $j$-invariant}

The next result gives exact and asymptotic formulas for the number of curves with a fixed \( j \)-invariant. Asymptotic estimates for \( \# \mathcal{E}_j(X; H^{\mathrm{cal}}) \) were previously given by the first author and Calvo-Monge in \cite[Theorem 1.6]{BSCM25}. In particular, for \( j \in \Q \smallsetminus \{0, 1728\} \), we improve upon those results by obtaining an explicit main term in the asymptotics. In these cases, Theorem 1.6 of \cite{BSCM25} only establishes the upper bound \( \# \mathcal{E}_j(X; H^{\mathrm{cal}}) = O(X^{1/6}) \).

\begin{theorem}\label{thm:fixed-j-theorem-generalized}
Let $j \in \Q$. Then for every $X > 1$, the numbers of elliptic curves in the families $\mathcal{E}_j(X;H_{\alpha,\beta})$ and $\widetilde{\mathcal{E}}_j(X;H_{\alpha,\beta})$ satisfy the following formulas.

\noindent For the family $\mathcal{E}_j(X; H_{\alpha, \beta})$ we have
\begin{align}\label{eqn:asymptotics-calibrated-naive-height}
\# \mathcal{E}_j(X; H_{\alpha, \beta}) = 
\begin{cases}
\dfrac{2}{\beta^{1/2}\zeta(6)}X^{1/2}+O(X^{1/12}) &\text{for $j = 0$,}\vspace{5pt}\\
\dfrac{2}{\alpha^{1/3}\zeta(4)}X^{1/3}+O(X^{1/12}) &\text{for $j = 1728$,} \vspace{5pt}\\
\dfrac{2c(j; H_{\alpha, \beta})}{\zeta(2)} X^{1/6} + O(X^{1/12}) &\text{for $j \neq 0, 1728$.}
\end{cases}
\end{align}

\noindent For the family $\widetilde{\mathcal{E}}_j(X; H_{\alpha, \beta})$ we have
\begin{align}\label{eqn:exact-formulas-calibrated-naive-height}
\# \widetilde{\mathcal{E}}_j(X; H_{\alpha, \beta}) = 
\begin{cases}
2\left\lfloor\dfrac{X^{1/2}}{\beta^{1/2}}\right\rfloor &\text{for $j = 0$,}\vspace{5pt}\\
2\left\lfloor\dfrac{X^{1/3}}{\alpha^{1/3}}\right\rfloor &\text{for $j = 1728$,} \vspace{5pt}\\
2\left\lfloor c(j; H_{\alpha, \beta}) X^{1/6}\right\rfloor &\text{for $j \neq 0, 1728$.}
\end{cases}
\end{align}
Here $c(j; H_{\alpha, \beta})$ is the explicit constant defined in equation \eqref{eqn:cj-constant-definition}.
\end{theorem}

\subsubsection{Counting formulas for CM elliptic curves}

As an application of the formulas for $\#\mathcal{E}_j(X;H_{\alpha,\beta})$ in Theorem~\ref{thm:fixed-j-theorem-generalized}, we obtain exact and asymptotic formulas for the number of CM elliptic curves ordered by the generalized naive height $H_{\alpha,\beta}$. 

In the special case of the calibrated naive height $H^{\mathrm{cal}}$, an asymptotic formula for
$\#\mathcal{E}^{\mathrm{cm}}(X;H^{\mathrm{cal}})$
was previously obtained in \cite[Theorem 1.3]{BSCM25}. 
Our result recovers this case and refines it by incorporating the improved asymptotics for $\#\mathcal{E}_j(X;H_{\alpha,\beta})$, which allow us to obtain a tertiary term in the expansion of $\#\mathcal{E}^{\mathrm{cm}}(X;H_{\alpha,\beta})$.

\begin{theorem}\label{thm:CM-theorem-generalized}
For every $X > 1$, the numbers of CM elliptic curves in the families $\mathcal{E}^{\mathrm{cm}}(X;H_{\alpha,\beta})$ and $\widetilde{\mathcal{E}}^{\mathrm{cm}}(X;H_{\alpha,\beta})$ satisfy the formulas
\begin{align}\label{eqn:Ecm-asymptotics}
\#\mathcal{E}^{\mathrm{cm}}(X;H_{\alpha,\beta}) = \frac{2}{\beta^{1/2} \zeta(6)} X^{1/2} + \frac{2}{ \alpha^{1/3}\zeta(4)} X^{1/3} +\frac{2}{\zeta(2)}K(\alpha,\beta)X^{1/6} + O(X^{1/12})
\end{align}
and
\begin{align}\label{eqn:Ecm-hat-asymptotics}
\#\widetilde{\mathcal{E}}^{\mathrm{cm}}(X;H_{\alpha,\beta}) = 2\left\lfloor\frac{X^{1/2}}{\beta^{1/2}}\right\rfloor+ 2\left\lfloor\frac{X^{1/3}}{\alpha^{1/3}}\right\rfloor  + \sum_{\substack{j \in \mathcal{J}^{\mathrm{cm}}\\ j \neq 0, 1728}} 2\left\lfloor c(j;H_{\alpha,\beta})X^{1/6} \right\rfloor,  
\end{align}
where
\begin{align}
K(\alpha,\beta) := \sum_{\substack{j \in\mathcal{J}^{\mathrm{cm}} \\ j \neq 0, 1728}}c(j;H_{\alpha,\beta}).
\end{align}
Here $c(j; H_{\alpha, \beta})$ is the constant defined in equation \eqref{eqn:cj-constant-definition}.
\end{theorem}

\begin{remark}
In \cite[\S 3.4]{Bal16}, Balakrishnan et al.\ state that there are \num{65732} isomorphism classes of elliptic curves \( E_{A, B} \) with complex multiplication having naive height up to \( 2.7 \cdot 10^{10} \). In our notation, this means that they computed the exact value \( \#\mathcal{E}^{\mathrm{cm}}(2.7 \cdot 10^{10}; H^{\mathrm{cal}}) = \num{65732} \). 

The approximation given by the first three terms in equation \eqref{eqn:Ecm-asymptotics} yields
\begin{align*}
\#\mathcal{E}^{\mathrm{cm}}(2.7 \cdot 10^{10}; H^{\mathrm{cal}}) &\approx \frac{2}{3^{3/2} \zeta(6)} (2.7 \cdot 10^{10})^{1/2} + \frac{2^{1/3}}{\zeta(4)} (2.7 \cdot 10^{10})^{1/3} +\frac{2C}{\zeta(2)}(2.7 \cdot 10^{10})^{1/6} \\
&\approx \num{65722.95},
\end{align*}
which corresponds to a relative error of only about \( 0.014\% \).
To further illustrate the improved results on the asymptotic estimates for $\# \mathcal{E}_j(X; H^{\mathrm{cal}})$, Table \ref{table:CM-exact-vs-asymptotic} presents the number of curves $\#\mathcal{E}^{\mathrm{cm}}(X; H^{\mathrm{cal}})$, with their corresponding approximation and relative error, for various values of $X$, including the value calculated by  \cite[p.~352]{Bal16}.
\begin{table}[H]
{\tabulinesep=1.2mm
\begin{tabu}{c c c c} \hline
$X$&$\#\mathcal{E}^{\mathrm{cm}}(X; H^{\mathrm{cal}})$& $\frac{2}{3^{3/2} \zeta(6)} X^{1/2} + \frac{2^{1/3}}{\zeta(4)} X^{1/3} +\frac{2C}{\zeta(2)}X^{1/6} $& Relative error\\ \hline
10&2& 5.40&$170\%$\\ 
$10^2$&6&11.68&$95\%$\\ 
$10^3$&24&27.26&$13\%$\\ 
$10^4$&66&68.28&$3.45\%$\\ 
$10^5$&180&181.55&$0.86\%$\\ 
$10^6$&508&506.31&$0.33\%$\\ 
$10^7$&1470&1464.17&$0.40\%$\\ 
$2.7\cdot 10^{10}$ &65 732& 65 722.95&$0.014\%$\\ \hline

\end{tabu}}
\caption{The numbers of curves $\#\mathcal{E}^{\mathrm{cm}}(X;H^{\mathrm{cal}})$ and the asymptotic approximation given by Theorem \ref{thm:CM-theorem-generalized} for the calibrated naive height, with their respective error.}
\label{table:CM-exact-vs-asymptotic}
\end{table}
\end{remark}

\begin{remark}
The error terms $O(X^{1/12})$ appearing in the asymptotic formulas from Theorems \ref{thm:fixed-j-theorem-generalized} and \ref{thm:CM-theorem-generalized} can be improved significantly under the assumption of the Riemann Hypothesis. The reason for this is that in our proofs we use an asymptotic formula for the number $Q_k(x)$ of $k$-free integers in the interval $[1, x]$, when approximated by the form
$$
Q_k(x) = \frac{x}{\zeta(k)} + E_k(x),
$$
where $E_k(x)$ is the absolute error. In Proposition \ref{prop: k-free-numbers-asymptotics}, we cite the unconditional bound
$$
E_k(x)=O(x^{1/k}).
$$
However, in this paper, we only use the cases for $k = 2, 4$ and $6$, and under the Riemann Hypothesis, the best known bounds for the error $E_k(x)$ for these values of $k$ are
\begin{align*}
E_2(x)=O(x^{11/35 +\varepsilon}),\quad
E_4(x)=O(x^{17/94+\varepsilon})\quad\text{ and }\quad
E_6(x)=O(x^{4/31+\varepsilon}),
\end{align*}
for any $\varepsilon > 0$. These are due to Liu \cite{Liu16, Liu14} for $k = 2, 6$ and to Baker and Powell \cite{BP10} for $k=4$ (see the introduction of \cite{MOT21} for a recent survey on these results). Table \ref{table:conditional-error-bounds} shows the improvements to the error terms if these conditional bounds are used in our proofs. Note that $\frac{1}{12} = 0.08\overline{3}$.

\begin{table}[H]
\centering
{\tabulinesep=1.2mm
\begin{tabu}{c|c|c}
\hline
Family of curves & Improved conditional error bound & Numerical value \\
\hline
$\# \mathcal{E}_0(X; H)$ & $O(X^{2/31 + \varepsilon})$ & $\frac{2}{31} \approx 0.0645161$ \\
$\# \mathcal{E}_{1728}(X; H)$ & $O(X^{17/282 + \varepsilon})$ & $\frac{17}{282} \approx 0.0602837$ \\
$\# \mathcal{E}_{j}(X; H)$ for $j \neq 0, 1728$ & $O(X^{11/210 + \varepsilon})$ & $\frac{11}{210} \approx 0.0523810$ \\
$\#\mathcal{E}^{\mathrm{cm}}(X; H)$ & $O(X^{2/31 + \varepsilon})$ & $\frac{2}{31} \approx 0.0645161$ \\
\hline
\end{tabu}}
\caption{Improved conditional error bounds for different families of curves. Here $H$ stands for any normalization of the naive height considered in this paper.}
\label{table:conditional-error-bounds}
\end{table}

\end{remark}

For ease of comparison, Table~\ref{table:main-terms-calibrated-uncalibrated}
summarizes the main asymptotic terms appearing in the preceding theorems for
the calibrated and uncalibrated naive heights. The table makes explicit that
the two normalizations lead to the same powers of $X$, but with different
height-dependent constants.

\begin{table}[H]
\centering
\renewcommand{\arraystretch}{1.6}
\begin{tabular}{c c c}
\toprule
\textbf{Family} 
&
\textbf{$H=H^{\mathrm{cal}}$}
&
\textbf{$H=H^{\mathrm{ncal}}$}
\\
\midrule

$\mathcal{E}(X;H)$
&
$\displaystyle \frac{2^{4/3}}{3^{3/2}\zeta(10)}X^{5/6}$
&
$\displaystyle \frac{4}{\zeta(10)}X^{5/6}$
\\[8pt]

$\mathcal{E}_0(X;H)$
&
$\displaystyle \frac{2}{3^{3/2}\zeta(6)}X^{1/2}$
&
$\displaystyle \frac{2}{\zeta(6)}X^{1/2}$
\\[8pt]

$\mathcal{E}_{1728}(X;H)$
&
$\displaystyle \frac{2^{1/3}}{\zeta(4)}X^{1/3}$
&
$\displaystyle \frac{2}{\zeta(4)}X^{1/3}$
\\[8pt]

$\mathcal{E}_j(X;H)$, $j\neq 0,1728$
&
$\displaystyle \frac{2c(j;H^{\mathrm{cal}})}{\zeta(2)}X^{1/6}$
&
$\displaystyle \frac{2c(j;H^{\mathrm{ncal}})}{\zeta(2)}X^{1/6}$
\\[8pt]

$\mathcal{E}^{\mathrm{cm}}(X;H)$
&
$\displaystyle 
\frac{2}{3^{3/2}\zeta(6)}X^{1/2}
+
\frac{2^{1/3}}{\zeta(4)}X^{1/3}
+
\frac{2C}{\zeta(2)}X^{1/6}$
&
$\displaystyle 
\frac{2}{\zeta(6)}X^{1/2}
+
\frac{2}{\zeta(4)}X^{1/3}
+
\frac{2D}{\zeta(2)}X^{1/6}$
\\

\bottomrule
\end{tabular}
\caption{Main terms in the asymptotic formulas for the calibrated and uncalibrated naive heights. Here $j\neq 0,1728$, and
\[
C=\sum_{\substack{j\in\mathcal{J}^{\mathrm{cm}}\\ j\neq 0,1728}}c(j;H^{\mathrm{cal}}) \approx 0.95058,
\qquad
D=\sum_{\substack{j\in\mathcal{J}^{\mathrm{cm}}\\ j\neq 0,1728}}c(j;H^{\mathrm{ncal}}) \approx 1.20947.
\]
}
\label{table:main-terms-calibrated-uncalibrated}
\end{table}

\subsection{Parametrization of Elliptic Curves with a given \texorpdfstring{$j$}{j}-invariant}
The following theorem provides the parametrization of the families $\widetilde{\mathcal{E}}_j(X;H_{\alpha,\beta})$ and $\mathcal{E}_j(X;H_{\alpha,\beta})$ of elliptic curves in short Weierstrass form with a fixed $j$-invariant, as was mentioned in the introduction. This parametrization gives all the curves as twists of the two curves of minimal naive height with the given $j$-invariant, as clarified in Remark \ref{remark:parametrization-remark}.

\begin{theorem}\label{thm:parametrization-and-exact-count-curves-with-j-introduction}
Let $j \in \mathbb{Q} \smallsetminus \{0,1728\}$. Then the set of all elliptic curves $E_{A, B}: y^2=x^3+Ax+B$ with $A, B \in \mathbb{Z}$ that have $j\left(E_{A, B}\right)=j$ and $H_{\alpha, \beta}\left(E_{A, B}\right) \leq X$ can be explicitly described parametrically by
\begin{align}\label{eqn:elliptic-j-inv-parametrization-introduction}
\widetilde{\mathcal{E}}_j\left(X ; H_{\alpha, \beta}\right) = \left\{ E_{A(j, m), B(j, m)} \suchthat \text {$m \in \mathbb{Z} \smallsetminus \{  0\}$ and $|m| \leq c(j; H_{\alpha, \beta}) X^{1/6}$} \right\}
\end{align}
where $A(j, m)$ and $B(j, m)$ are the integers defined by
\begin{align}
A(j,m)=\frac{N_{a(j)}^2 m^2}{a(j)}\quad \text{and} \quad  B(j,m)=\frac{N_{a(j)}^3 m^3}{a(j)}
\end{align}
and the constant $c(j; H_{\alpha, \beta})$ is defined by
\begin{align}
c(j; H_{\alpha, \beta}) := \min \left\{\frac{|a(j)|^{1 / 2}}{N_{a(j)} \alpha^{1 / 6}}, \frac{|a(j)|^{1 / 3}}{N_{a(j)} \beta^{1 / 6}}\right\}.
\end{align}
This implies that
\begin{align}\label{eqn:exact-j-count-generalized-height-formula-introduction}
\# \widetilde{\mathcal{E}}_j\left(X ; H_{\alpha, \beta}\right)=2\left \lfloor c(j; H_{\alpha, \beta}) X^{1 / 6}\right \rfloor.
\end{align}
Moreover, for any $m \in \Z \smallsetminus \{ 0 \}$ and any prime $p \in \mathbb{P}$ we have that 
\[
2\ord_p(m) \leq \ord_p{(A(j, m))} \leq 2\ord_p(m)+1 \quad \text{if $\ord_p{(a(j))} \geq 0$}
\]  
and 
\[
3\ord_p(m) \leq \ord_p{(B(j, m))} \leq 3\ord_p(m)+2 \quad \text{if $\ord_p{(a(j))} < 0$}.
\] 
In particular, this implies that the curve $E_{A(j,m),B(j,m)}\in\mathcal{E}_j(X;H_{\alpha,\beta})$ if and only if $m$ is square-free and $|m| \leq c(j; H_{\alpha, \beta}) X^{1/6}$, in other words, we have
\begin{align}\label{eqn:Q-isomorphism-classes-elliptic-j-inv-parametrization-introduction}
\mathcal{E}_j\left(X ; H_{\alpha, \beta}\right) = \left\{ E_{A(j, m), B(j, m)} \suchthat \text {$m \in \mathbb{Z} \smallsetminus \{  0\}$ is square-free and $|m| \leq c(j; H_{\alpha, \beta}) X^{1/6}$} \right\}.
\end{align}
\end{theorem}

\begin{remark}\label{remark:parametrization-remark}
It is worth mentioning that in the literature there are explicit equations for curves in short Weierstrass form having a given $j$-invariant. For instance, if $j \in \Q \smallsetminus \{ 0, 1728 \}$, then the curve 
$$
E \colon y^2 = x^3 + \frac{3j}{1728 - j} x + \frac{2j}{1728 - j}
$$
has $j$-invariant $j(E) = j$ (see e.g. \cite[p. 47]{Was08}). However, for most values of $j$, the coefficients of $E$ are in $\Q \smallsetminus \Z$. Thus, our parametrization has the advantage that the coefficients of the curves are integers and moreover, for any given $j$, the curves of minimal naive height $H_{\alpha, \beta}$ correspond to the values $m = \pm 1$. Hence, the curves in our parametrization can be thought of as the quadratic twists of the curves of minimal naive height for a fixed value of $j$. 
\end{remark}

\subsection{Asymptotic density of the representatives for the \texorpdfstring{$\Q$}{Q}-isomorphism classes of elliptic curves over \texorpdfstring{$\Q$}{Q}}

The counting formulas established in this paper allow us to determine the asymptotic density of $ \mathbb{Q} $-isomorphism classes of elliptic curves over \( \mathbb{Q} \) within the families considered, when the curves are ordered by the generalized naive height \( H_{\alpha, \beta} \).

\begin{theorem}\label{thm:distribution-reps-all-curves}
As $X\to\infty$, the density of the representatives for the $\Q$-isomorphism classes of elliptic curves over $\Q$ is given by
\begin{align}
\frac{\# \mathcal{E}(X; H_{\alpha,\beta})}{\# \widetilde{\mathcal{E}}(X; H_{\alpha,\beta})} =
\frac{1}{\zeta(10)} + O(X^{-1/3}).
\end{align}
\end{theorem}

\begin{remark}
Observe that $\dfrac{1}{\zeta(10)}\approx 99.9\%$, so that about $99.9\%$ of the curves in the family $\widetilde{\mathcal{E}}$ consist of representatives for the $\Q$-isomorphism classes of elliptic curves over $\Q$.
\end{remark}

The next theorem shows that when $j \in \Q \smallsetminus \{ 0, 1728 \}$, all the classes are equidistributed.

\begin{theorem}\label{thm:distribution-reps-fixed-j}
Let $j \in \Q$. Then, as $X \to \infty$, the density of the representatives for the $\Q$-isomorphism classes of elliptic curves over $\Q$ with fixed $j$-invariant is given by
\begin{align}
\frac{\# \mathcal{E}_j(X; H_{\alpha,\beta})}{\# \widetilde{\mathcal{E}}_j(X; H_{\alpha,\beta})} =
\begin{cases}
\dfrac{1}{\zeta(6)} + O(X^{-5/12}) &\text{for $j = 0$,}\vspace{5pt}\\
\dfrac{1}{\zeta(4)} + O(X^{-1/4}) &\text{for $j = 1728$,} \vspace{5pt}\\
\dfrac{1}{\zeta(2)} + O(X^{-1/12}) &\text{for $j \neq 0, 1728$.}
\end{cases}
\end{align}
\end{theorem}
\begin{remark}
Observe that $\dfrac{1}{\zeta(6)} \approx 98.3\%$,  $\dfrac{1}{\zeta(4)} \approx 92.4\%$ and $\dfrac{1}{\zeta(2)} \approx 60.8\%$. Thus, excluding the exceptional cases of $j = 0$ and $j = 1728$, for every other rational $j$-invariant, about $60.8\%$ of the curves in the family $\widetilde{\mathcal{E}}_j$ correspond to representatives for the $\Q$-isomorphism classes of elliptic curves $E/\Q$ with $j(E) = j$.
\end{remark}

\subsection{Organization of the paper}

The paper is organized as follows. In Section~\ref{section: cuspidal cubic count}, we prove a proposition that provides a parametrization of the set of integral points lying on a cuspidal cubic. From this parametrization, we derive an exact formula for the number of such points. This result serves as a key tool in establishing our exact and asymptotic formulas for counting elliptic curves. These formulas are proved in Section~\ref{section:distribution-of-cm-elliptic-curves-over-E}, where, after introducing some preparatory lemmas, we state and prove our main results concerning the generalized naive height \( H_{\alpha, \beta} \). Finally, in Section~\ref{section:explicit-computations}, we present explicit computations and numerical data related to CM elliptic curves, as a computational complement to the theoretical results developed earlier.

%%%%%%%%%%%%%%%%%%%%%%%%%%%%%%%%%%%%%%%%%%%%%%%%%%%
%%%%%%%%%%%%%%%%%%%%%%%%%%%%%%%%%%%%%%%%%%%%%%%%%%%
%%% Integral points on cuspidal cubics section %%%%
%%%%%%%%%%%%%%%%%%%%%%%%%%%%%%%%%%%%%%%%%%%%%%%%%%%
%%%%%%%%%%%%%%%%%%%%%%%%%%%%%%%%%%%%%%%%%%%%%%%%%%%

\section{Integral Points on Cuspidal Cubics}\label{section: cuspidal cubic count}

The following result provides an exact formula for the number of integral points inside a box in \( \R^2 \) that lie on a cuspidal cubic. This count plays a key role in deriving both exact and asymptotic formulas for the number of elliptic curves in the various families considered in this paper. The proposition also refines Lemma 2.3 of \cite{BSCM25}, which only gives an upper bound.

\begin{proposition}\label{prop:cuspidal-cubic-exact-count}
Let \(a \in \mathbb{Q}^{\times}\) and consider the cuspidal cubic $\mathcal{C}_{a}: y^{2}=a x^{3}$. Moreover, let $\mathcal{C}_{a}(\mathbb{Z})$ be the set of integral points on $\mathcal{C}_{a}$, i.e., the set of points \(\left(x_{0}, y_{0}\right) \in \mathbb{Z}^{2}\) such that \(y_{0}{ }^{2}=a x_{0}{ }^{3}\) and similarly, for $T_1 > 0$ and $T_2 > 0$, let
\[
\mathcal{C}_{a}(\mathbb{Z} ; T_1, T_2):=\left\{\left(x_{0}, y_{0}\right) \in \mathcal{C}_{a}(\mathbb{Z}) \suchthat \text{$|x_{0}| \leq T_1$ and $|y_0| \leq T_2$ }\right\}.
\]
Now, we let
\[
a=(-1)^{\varepsilon} \prod_{p \in \mathbb{P}} p^{\operatorname{ord}_{p}(a)}
\]
be the prime factorization of $a$. Also, for each \(p \in \mathbb{P}\), define the integer
\[
\alpha_{p}(a):= \begin{cases} 
\left \lceil\dfrac{\operatorname{ord}_{p}(a)}{2} \right \rceil & \text{if $\operatorname{ord}_{p}(a) \geq 0$,} \vspace{5pt}\\ \left \lceil\dfrac{\operatorname{ord}_{p}(a)}{3} \right \rceil & \text{if $\operatorname{ord}_{p}(a) < 0$.}
\end{cases}
\]
Finally, let $\displaystyle{N_{a}:=\prod \limits_{p \in \mathbb{P}} p^{\alpha_{p}(a)} \in \mathbb{Q}^{\times}}$. Then there is a bijection

\[
\begin{aligned}
\varphi_{a} : N_{a} \mathbb{Z} & \longrightarrow \mathcal{C}_{a}(\mathbb{Z}) \\
t & \longmapsto\left(\frac{t^{2}}{a}, \frac{t^{3}}{a}\right)
\end{aligned}
\]
and this bijection implies that for $T_1 > 0$ and $T_2 > 0$ we have
\[
\mathcal{C}_a(\Z;T_1,T_2)=\left\{\left(\frac{N^2_am^2}{a},\frac{N^3_am^3}{a}\right)\suchthat m\in\Z, |m|\leq \min{\left\{\frac{|a|^{1/2}}{N_a}T^{1/2}_1, \frac{|a|^{1/3}}{N_a}T^{1/3}_2\right\}}\right\}
\]
and therefore
\[
\# \mathcal{C}_{a}(\mathbb{Z}; T_1, T_2) = 2\left\lfloor \min{ \left \{ \frac{|a|^{1/2}}{N_{a}} T_1^{1/2}, \frac{|a|^{1/3}}{N_{a}} T_2^{1/3} \right\}} \right\rfloor+1. 
\]

\end{proposition}

\begin{proof}
The map
\begin{align*}
\widetilde{\varphi}_a \colon &\mathbb{R} \longrightarrow \mathcal{C}_a(\mathbb{R})\\
&t \longmapsto\left(\frac{t^2}{a}, \frac{t^3}{a}\right)
\end{align*}
is a bijection, obtained by projecting from the singular point at the origin on the cuspidal cubic $\mathcal{C}_a(\R)$ onto the vertical line $x = 1$. We claim that its restriction to $\mathbb{Q}$, namely the map $\widetilde{\varphi}_a |_{\mathbb{Q}}: \mathbb{Q} \longrightarrow \mathcal{C}_a(\mathbb{Q})$, is also a bijection. Indeed, it is clear that if $t \in \mathbb{Q}$, then $\widetilde{\varphi}_a(t) \in \mathcal{C}_a(\mathbb{Q})$. On the other hand, if $(x, y) \in \mathcal{C}_a(\mathbb{Q})$, then there is a $t \in \mathbb{R}$ such that $\widetilde{\varphi}_a(t) := \left(\frac{t^2}{a}, \frac{t^3}{a}\right)=(x, y) \in \mathbb{Q}^2$. If $(x, y) \neq(0,0)$, then $t \neq 0$ and both $x \neq 0$ and $y \neq 0$. Suppose indeed that $(x, y) \neq (0, 0)$. Hence
$$
t=\frac{\dfrac{t^3}{a}}{\dfrac{t^2}{a}} = \frac{y}{x} \in \mathbb{Q}^{\times}.
$$

We will show that the map
\begin{align*}
\varphi_a \colon &N_a \mathbb{Z} \longrightarrow \mathcal{C}_a(\mathbb{Z})\\
&\quad t \longmapsto\left(\frac{t^2}{a}, \frac{t^3}{a}\right)
\end{align*}
is a bijection. First let $t := N_a m \in N_a \Z$ with $m \in \Z$. Now, let $p \in \mathbb{P}$ be an arbitrary prime number. We will check that $\ord_p{(t^2/a)} \geq 0$ and $\ord_p{(t^3/a)} \geq 0$ and this will imply that $(t^2/a, t^3/a) \in \mathcal{C}_a(\Z)$. Note that
\begin{align}
\operatorname{ord}_p\left(\frac{t^2}{a}\right) &= 2 \operatorname{ord}_p(t)-\operatorname{ord}_p(a) \geq 0 \iff \operatorname{ord}_p(t) \geq \frac{\operatorname{ord}_p(a)}{2}, \label{eqn:order-inequality-1}\\
\operatorname{ord}_p\left(\frac{t^3}{a}\right) &= 3 \operatorname{ord}_p(t)-\operatorname{ord}_p(a) \geq 0 \iff \operatorname{ord}_p(t) \geq \frac{\operatorname{ord}_p(a)}{3}.\label{eqn:order-inequality-2}
\end{align}

Now, there are two cases to consider. First, if $\ord_p(a) \geq 0$ we have $\dfrac{\operatorname{ord}_p(a)}{2} \geq \dfrac{\operatorname{ord}_p(a)}{3}$, hence it is enough to show that $\operatorname{ord}_p(t) \geq \dfrac{\operatorname{ord}_p(a)}{2}$. Then, since $t = N_a m$, we have
\begin{align*}
\ord_p(t) &= \ord_p(N_a m) \\
&= \ord_p(N_a) + \ord_p(m)\\
&= \alpha_p(a) + \ord_p(m)\\
&\geq \alpha_p(a)\\
&= \left\lceil\frac{\operatorname{ord}_p(a)}{2}\right\rceil \\
&\geq \frac{\operatorname{ord}_p(a)}{2}.
\end{align*}
Next, if $\ord_p(a) < 0$, then $\dfrac{\operatorname{ord}_p(a)}{3}>\dfrac{\operatorname{ord}_p(a)}{2}$. Thus it is enough to show that $\operatorname{ord}_p(t) \geq \dfrac{\operatorname{ord}_p(a)}{3}$. Therefore, again since $t = N_a m$, a similar calculation as above shows that
\begin{align*}
\ord_p(t) &\geq \alpha_p(a)\\
&= \left\lceil\frac{\operatorname{ord}_p(a)}{3}\right\rceil \\
&\geq \frac{\operatorname{ord}_p(a)}{3}.
\end{align*}
This proves that $(t^2/a, t^3/a) \in \mathcal{C}_a(\Z)$ and hence $\varphi_a\left(N_a \mathbb{Z}\right) \subseteq \mathcal{C}_a(\mathbb{Z})$. This map is injective since it is the restriction of $\widetilde{\varphi}_a$.

Conversely, suppose that $(x, y) \in \mathcal{C}_a(\mathbb{Z}) \smallsetminus \{(0,0)\}$. Let $t \in \mathbb{Q}^{\times}$ be such that $\varphi_a(t)=(x, y)$.
Then $x=\dfrac{t^2}{a}, y=\dfrac{t^3}{a}$ and then $t=\dfrac{y}{x}$.
We claim that $t \in N_a \mathbb{Z}$; this
means that $\operatorname{ord}_p(t) \geq\operatorname{ord}_p\left(N_a\right)$ for every $p \in \mathbb{P}$. First, if $\ord_p(a) \geq 0$ we have $\ord_p(N_a) := \alpha_p(a) = \left\lceil\dfrac{\operatorname{ord}_p(a)}{2}\right\rceil$. Moreover, since $\operatorname{ord}_p(x) \geq 0$ and $\operatorname{ord}_p(y) \geq 0$, then \eqref{eqn:order-inequality-1} implies that 
$$
\operatorname{ord}_p(t) \geq \frac{\operatorname{ord}_p(a)}{2}
$$
and since $\ord_p(t) \in \Z$, it follows that 
$$
\operatorname{ord}_p(t) \geq\left\lceil\frac{\operatorname{ord}_p(a)}{2}\right\rceil=\operatorname{ord}_p\left(N_a\right).
$$
The case when $\ord_p(a) < 0$ follows from an analogous argument with the inequality \eqref{eqn:order-inequality-2}. This completes the proof that the map $\varphi_a: N_a \mathbb{Z} \longrightarrow \mathcal{C}_a(\mathbb{Z})$ is a bijection.

Now, using this bijection we want to count
\begin{align*}
\# \mathcal{C}_a(\mathbb{Z} ; T_1, T_2) := \left\{(x, y) \in \mathcal{C}_a(\mathbb{Z}) \suchthat |x|  \leq T_1 \text{ and } |y|\leq T_2 \right\}.
\end{align*}
Let $T_1 > 0$ and $T_2 > 0$. Then, the bijection $\varphi_a$ allows us to write
\begin{align*}
\mathcal{C}_a(\mathbb{Z} ; T_1, T_2) & =\left\{ \left(\frac{N^2_am^2}{a},\frac{N^3_am^3}{a}\right) \suchthat m\in\Z, \left| \frac{N^2_am^2}{a}\right| \leq T_1 \text{ and } \left|\frac{N^3_am^3}{a}\right|\leq T_2 \right\} \\
& =\left\{ \left(\frac{N^2_am^2}{a},\frac{N^3_am^3}{a}\right) \suchthat m\in\Z, |m| \leq \frac{|a|^{1/2}}{N_a}T^{1/2}_1 \text{ and } |m|\leq \frac{|a|^{1/3}}{N_a}T^{1/3}_2 \right\}\\
& =\left\{ \left(\frac{N^2_am^2}{a},\frac{N^3_am^3}{a}\right) \suchthat m\in\Z, |m| \leq \min\left\{\frac{|a|^{1/2}}{N_a}T^{1/2}_1, \frac{|a|^{1/3}}{N_a}T^{1/3}_2 \right\}\right\}.
\end{align*}
Finally, this immediately implies that
\[
\# \mathcal{C}_{a}(\mathbb{Z}; T_1, T_2) = 2\left\lfloor \min{ \left \{ \frac{|a|^{1/2}}{N_{a}} T_1^{1/2}, \frac{|a|^{1/3}}{N_{a}} T_2^{1/3} \right\}} \right\rfloor+1.
\]
\end{proof}

%%%%%%%%%%%%%%%%%%%%%%%%%%%%%%%%%%%%%%%%%%%%%%%%%%%%%%%%%%%%%
%%%%%%%%%%%%%%%%%%%%%%%%%%%%%%%%%%%%%%%%%%%%%%%%%%%%%%%%%%%%%
%%%%%%%%%%%%%%% Proofs section %%%%%%%%%%%%%%%%%%%
%%%%%%%%%%%%%%%%%%%%%%%%%%%%%%%%%%%%%%%%%%%%%%%%%%%%%%%%%%%%%
%%%%%%%%%%%%%%%%%%%%%%%%%%%%%%%%%%%%%%%%%%%%%%%%%%%%%%%%%%%%%
\section{Proofs of the theorems stated in Section \ref{section: definitions-and-results}}\label{section:distribution-of-cm-elliptic-curves-over-E}
In this section we carry out the proofs of the main theorems stated in Section \ref{section: definitions-and-results} of the paper. Throughout this section, $\mu(n)$ denotes the value of the Möbius function at an integer $n \geq 1$. We start by proving a couple of lemmas that we will need later.

\begin{lemma}\label{lemma:moebius-sum-lemma}
Let $s > 1$ and $x \geq 1$. Then
\begin{align}
\sum_{1 \leq n \leq x} \frac{\mu(n)}{n^s} = \frac{1}{\zeta(s)} + O(x^{1-s}).
\end{align}
\end{lemma}

\begin{proof}
We have
\begin{align*}
\sum_{1 \leq n \leq x} \frac{\mu(n)}{n^s} & =\sum_{n=1}^{\infty} \frac{\mu(n)}{n^s} - \sum_{n > x} \frac{\mu(n)}{n^s} \\
& =\frac{1}{\zeta(s)}-\sum_{n > x} \frac{\mu(n)}{n^s},
\end{align*}
where we used the well-known formula
$$
\frac{1}{\zeta(s)}=\sum_{n=1}^{\infty} \frac{\mu(n)}{n^s}
$$
for the reciprocal of the Riemann zeta function. Then, using the asymptotic bound (see e.g. \cite[Theorem 3.2]{Apostol76})
$$
\sum_{n>x} \frac{1}{n^s}=O(x^{1-s}), \quad \text { for } s>1 \text { and } x \geq 1,
$$
we see that
$$
\left|\sum_{n > x} \frac{\mu(n)}{n^s}\right| \leq \sum_{n > x} \frac{|\mu(n)|}{n^s} \leq \sum_{n > x} \frac{1}{n^s}=O(x^{1-s}).
$$

\end{proof}

\begin{lemma}\label{lemma:no-singular-elliptic-curve-j-inv}
Let $j \in \Q \smallsetminus \{ 0, 1728 \}$ and $(A, B) \in \Z^2 \smallsetminus \{ (0, 0) \}$. Then the curve $E_{A, B} \colon y^2 = x^3 + Ax + B$ is nonsingular and has $j$-invariant $j(E_{A, B}) = j$ if and only if the point $(A, B)$ lies on the cuspidal cubic $\mathcal{C}_{a(j)} \colon y^2 = a(j) x^3$, where $a(j) = \dfrac{4(1728 - j)}{27j} \in \Q^{\times}$. Moreover, the curve $E_{A, B}$ is singular if and only if the point $(A, B)$ lies on the cuspidal cubic $\mathcal{C}_{-\frac{4}{27}} \colon y^2 = -\frac{4}{27} x^3$.
\end{lemma}
\begin{proof}
Recall that the $j$-invariant of a nonsingular curve $E_{A,B}$ is given by 
\[
j(E_{A,B})=1728\dfrac{4A^3}{4A^3+27B^2}.
\]

\noindent $(\implies)$
If $E_{A, B}$ is nonsingular and $j(E_{A, B}) = j$, we can rewrite this condition as
\[
j(E_{A,B})=j\iff 1728\dfrac{4A^3}{4A^3+27B^2}=j\iff B^2=\dfrac{4(1728-j)}{27j}A^3.
\]
Thus, for a nonsingular curve $E_{A,B}$ to have $j$-invariant equal to $j$, for $j\neq 0,1728$, is equivalent to the point $(A,B) \in \Z^2 \smallsetminus \{ (0, 0) \}$ lying on the cuspidal cubic $\mathcal{C}_{a(j)} \colon y^2 = a(j) x^3$, where $a(j) := \dfrac{4(1728-j)}{27j}$. 

\noindent $(\impliedby)$
Now, a curve $E_{A,B}$ is singular if and only if $\Delta_{E_{A,B}}=-16(4A^3+27B^2)=0$. This condition can also be rewritten as
\[
\Delta_{E_{A,B}} = -16(4A^3 + 27B^2) = 0 \iff B^2 = -\dfrac{4}{27}A^3.
\]
Thus, a curve $E_{A,B}$ is singular if and only if the point $(A,B)$ lies on the cuspidal cubic $\mathcal{C}_{-\frac{4}{27}} \colon y^2 = -\frac{4}{27} x^3$. A quick verification shows that no $j\in\Q\smallsetminus\{0,1728\}$ satisfies $a(j) = -\frac{4}{27}$, and since two cuspidal cubics $\mathcal{C}_a$ and $\mathcal{C}_b$ with $a \neq b$ intersect in the affine plane $\mathbb{A}^2(\R)$ only at the origin, then if the point $(A,B)\in \Z^2 \smallsetminus \{ (0, 0) \}$ lies on the cuspidal cubic $C_{a(j)}$, it does not lie on the cubic $C_{-\frac{4}{27}}$. Hence the curve $E_{A, B}$ is nonsingular and has $j$-invariant $j(E_{A, B}) = j$.
\end{proof}

The following lemma gives an explicit form of a result stated by Brumer \cite[p. 455]{Bru92}. In particular, the lemma gives a decomposition of the family of elliptic curves $\widetilde{\mathcal{E}}(X; H_{\alpha,\beta})$ as a finite disjoint union of the twists of families of representatives of the $\Q$-isomorphism classes of elliptic curves in short Weierstrass form. We also decompose the family of elliptic curves with a given $j$-invariant, and provide an explicit construction of the twists, which will prove useful for finding the \textit{unique} $\Q$-isomorphism class representative of a given curve $E_{A,B}$ (see Remark \ref{remark:construction-Q-isomorphism-representative}).

\begin{lemma}\label{lemma: elliptic-family-decomposition}
For $d \in \Z$, we define $d * E_{A, B} := E_{d^4A, d^6 B}$. Then, every curve $E_{A', B'} \in \widetilde{\mathcal{E}}(X;H_{\alpha,\beta})$ can be written uniquely as a twist $d * E_{A, B}$ for a unique curve $E_{A, B} \in \mathcal{E}(d^{-12} X; H_{\alpha,\beta})$; the unique integer $d \in \Z_{\geq 1}$ is given by
\begin{align}\label{eqn:d-twist-decomposition}
d := \prod_{p \in \mathbb{P}} p^{\gamma_p(A', B')},
\end{align}
where $\gamma_p(A', B') := \max{\{ \gamma \in \Z_{\geq 0} \colon p^{4\gamma} \mid A' \text{ and } p^{6\gamma} \mid B' \}}$ and it satisfies that $1 \leq d \leq \lfloor X^{1/12} \rfloor$. In particular, this implies that we have the decomposition
\begin{align}\label{eqn:curve-families-decomposition}
\widetilde{\mathcal{E}}(X; H_{\alpha,\beta}) = \coprod_{d = 1}^{\lfloor X^{1/12} \rfloor} d * \mathcal{E}(d^{-12} X;H_{\alpha,\beta}). 
\end{align}
Moreover, for a given $j\in\Q$, we also have the decomposition
\begin{align}\label{eqn: curve-families-decomposition-j invariant}
\widetilde{\mathcal{E}}_j(X; H_{\alpha,\beta}) = \coprod_{d = 1}^{\lfloor X^{1/12} \rfloor} d * \mathcal{E}_j(d^{-12} X;H_{\alpha, \beta}). 
\end{align}
\end{lemma}
\begin{proof}
Let $E_{A^{\prime},B^{\prime}}\in\widetilde{\mathcal{E}}(X;H_{\alpha,\beta})$, and let $d\in\Z_{\geq 1}$ as in \ref{eqn:d-twist-decomposition}. For the curve $E_{A^{\prime},B^{\prime}}$, take $A=A^{\prime}/d^4$ and $B=B^{\prime}/d^6$. Then, by definition, it follows that $E_{A^{\prime},B^{\prime}}=d*E_{A,B}$ and, since from the definition of the naive height, we have
\[
H_{\alpha,\beta}(d * E_{A^{\prime}, B^{\prime}}) = d^{12} \cdot H_{\alpha,\beta}(E_{A^{\prime}, B^{
\prime
}}),
\]
therefore, $H_{\alpha,\beta}(d * E_{A^{\prime}, B^{\prime}}) \leq X$ if and only if $H_{\alpha,\beta}(E_{A^{\prime}, B^{\prime}}) \leq X/d^{12}$. By the definition of $d$, for each prime $p \in \mathbb{P}$, the maximality of $\ord_p(d) := \gamma_p(A', B')$ implies that $p^4 \nmid A$ or $p^6 \nmid B$, so that $E_{A,B}\in\mathcal{E}(d^{-12}X;H_{\alpha,\beta})$.
Next, to prove that the curve $E_{A,B}$ is unique in $\mathcal{E}(d^{-12}X;H_{\alpha,\beta})$, suppose for a contradiction, that there exists another curve $E_{A_0,B_0} \in \mathcal{E}(d^{-12}X;H_{\alpha,\beta})$ and $d_0\in\Z_{\geq 1}$ such that $E_{A^{\prime},B^{\prime}}=d_0*E_{A_0,B_0}$. Then, we have
\begin{align}\label{eqn:twist-uniqueness}
\begin{cases}
d_0^4A_0=d^4A,\\
d_0^6B_0=d^6B.
\end{cases}
\end{align}
We want to prove that $d=d_0$. Without loss of generality, suppose that $d_0\geq d$. Now suppose, in order to reach a contradiction, that $d_0>d$. This implies that $0<\dfrac{d}{d_0}<1$, so that $\dfrac{d}{d_o}\not\in\Z$. Therefore, there exists a prime $p$ and  an integer $k \in \Z_{\geq 1}$ such that $\operatorname{ord}_{p}\left(\dfrac{d}{d_0}\right)=-k\leq -1$. Then, from equation \eqref{eqn:twist-uniqueness}, it follows that for this prime $p$ we have
\[
\begin{cases}
0\leq \operatorname{ord}_{p}(A_0)=\ord_p\bigg(\left(\dfrac{d}{d_0}\right)^4A\bigg)=4\operatorname{ord}_{p} \left(\dfrac{d}{d_0} \right)+\operatorname{ord}_{p}(A) \implies 4k\leq \operatorname{ord}_{p}(A) \vspace{4pt}\\
0\leq \operatorname{ord}_{p}(B_0)=\ord_p\bigg(\left(\dfrac{d}{d_0}\right)^6B\bigg)=6\operatorname{ord}_{p} \left(\dfrac{d}{d_0} \right)+\operatorname{ord}_{p}(B) \implies 6k\leq \operatorname{ord}_{p}(B)
\end{cases}
\]
It follows that $p^{4k}\mid A$ and $p^{6k}\mid B$, which is a contradiction. Hence, $d=d_0$, and $A_0=A$ and $B_0=B$.
Finally, notice that $H_{\alpha,\beta}(E_{A, B})\geq 1$, so that 
\[
1\leq d^{12}\leq d^{12}H_{\alpha,\beta}(E_{A, B})=H_{\alpha,\beta}(E_{A', B'})\leq X.
\]
We conclude that 
\[
1\leq d\leq \lfloor X^{1/12}\rfloor.
\]
These facts prove the decomposition \eqref{eqn:curve-families-decomposition}. To prove \eqref{eqn: curve-families-decomposition-j invariant}, simply notice that the $j$-invariant of a curve is invariant under the twist operator.
\end{proof}

We will also need the following generalized form of the Möbius inversion formula (see e.g. \cite[Theorem 2.23]{Apostol76}).

\begin{proposition}[Generalized Möbius inversion formula]\label{prop:moebius-inversion-formula}
Let $F(x)$ and $G(x)$ be complex-valued functions defined in the interval $[1, \infty)$, then
$$
G(x) = \sum_{1 \leq d \leq x} F \left( \frac{x}{d} \right)
$$
for every $x \geq 1$ if and only if
$$
F(x) = \sum_{1 \leq d \leq x} \mu(d) G \left( \frac{x}{d} \right)
$$
for every $x \geq 1$.
\end{proposition}

We are now ready to prove the main theorems from the introduction. We start by proving Theorem \ref{thm:global-counts-generalized-intro}. For this we prove the following version for the generalized naive height $H_{\alpha, \beta}$.

\begin{theorem}\label{thm:global-counts-generalized}
For every $X > 1$, the number of elliptic curves in the families $\mathcal{E}(X; H_{\alpha, \beta})$ and $\widetilde{\mathcal{E}}(X; H_{\alpha, \beta})$ satisfy the formulas
\begin{align}\label{eqn:E-global-formula-generalized}
\#\mathcal{E}(X; H_{\alpha, \beta}) = \frac{4}{\alpha^{1/3} \beta^{1/2} \zeta(10)} X^{5/6}+O(X^{1/2})
\end{align}
and
\begin{align}\label{eqn:widetilde-E-global-formula-generalized-proof-section}
\#\widetilde{\mathcal{E}}(X; H_{\alpha, \beta}) &= \left( 2 \left\lfloor \frac{X^{1/3}}{\alpha^{1/3}} \right\rfloor + 1 \right) \left( 2 \left\lfloor \frac{X^{1/2}}{\beta^{1/2}} \right\rfloor + 1 \right) - 2 \left\lfloor c(\alpha, \beta) X^{1/6} \right\rfloor - 1 \\
&= \frac{4}{\alpha^{1/3} \beta^{1/2}}X^{5/6}+O(X^{1/2}),
\end{align}
where 
\begin{align}\label{eqn:c-alpha-beta-constant-definition-proof-section}
c(\alpha, \beta) := \min{ \left\{  \frac{1}{3^{1/2} \alpha^{1/6}}, \frac{1}{2^{1/3} \beta^{1/6}} \right\} }.
\end{align}
\end{theorem}

\begin{proof}
We start by proving the exact formula \eqref{eqn:widetilde-E-global-formula-generalized-proof-section}. Then we will use it to prove the asymptotic formula \eqref{eqn:E-global-formula-generalized}. First, observe that $H_{\alpha, \beta}(E_{A, B}) \leq X$ if and only if $|A| \leq \dfrac{X^{1/3}}{\alpha^{1/3}}$ and $|B| \leq \dfrac{X^{1/2}}{\beta^{1/2}}$. Then the total number of curves in $\widetilde{\mathcal{E}}(X; H_{\alpha, \beta})$ is equal to the number of lattice points $(A, B) \in \Z^2$ that lie inside the rectangular box
$$
\mathcal{B}(X; \alpha, \beta) := \left\{ (x, y) \in \R^2 \suchthat \text{$|x| \leq \frac{X^{1/3}}{\alpha^{1/3}}$ and $|y| \leq \frac{X^{1/2}}{\beta^{1/2}}$} \right\},
$$
minus the number of lattice points $(A, B) \in \Z^2$ that lie on the cuspidal cubic $\mathcal{C}_{-\frac{4}{27}} \colon y^2 = -\frac{4}{27} x^3$, because these last points correspond to the singular curves inside the rectangular box, by Lemma \ref{lemma:no-singular-elliptic-curve-j-inv}.

Now, we have
\begin{align}\label{eqn:rectangular-box-count}
\mathcal{B}(X; \alpha, \beta) \cap \Z^2 = \left( 2 \left\lfloor \frac{X^{1/3}}{\alpha^{1/3}} \right\rfloor + 1 \right) \left( 2 \left\lfloor \frac{X^{1/2}}{\beta^{1/2}} \right\rfloor + 1 \right).
\end{align}
On the other hand, by Proposition \ref{prop:cuspidal-cubic-exact-count}, the number of lattice points inside the rectangular box that also lie on the cuspidal cubic $\mathcal{C}_{-\frac{4}{27}}$ is equal to
\begin{align}\label{eqn:singular-curves-count}
\# \mathcal{C}_{-\frac{4}{27}} \left(\Z; \frac{X^{1/3}}{\alpha^{1/3}}, \frac{X^{1/2}}{\beta^{1/2}} \right) &= 2 \left\lfloor \min{ \left\{  \frac{|\frac{-4}{27}|^{1/2}}{N_{-\frac{4}{27}}} \frac{X^{1/6}}{\alpha^{1/6}}, \frac{|\frac{-4}{27}|^{1/3}}{N_{-\frac{4}{27}}} \frac{X^{1/6}}{\beta^{1/6}} \right\} }  \right\rfloor + 1 \notag\\
&= 2 \left\lfloor \min{ \left\{  \frac{1}{3^{1/2} \alpha^{1/6}}, \frac{1}{2^{1/3} \beta^{1/6}} \right\} } X^{1/6}  \right\rfloor + 1,
\end{align}
where we have calculated the values
$$
N_{-\frac{4}{27}} = \frac{2}{3}, \quad \frac{|\frac{-4}{27}|^{1/2}}{N_{-\frac{4}{27}}} = \frac{1}{3^{1/2}} \quad \text{and} \quad \frac{|\frac{-4}{27}|^{1/3}}{N_{-\frac{4}{27}}} = \frac{1}{2^{1/3}}.
$$
Thus, combining equations \eqref{eqn:rectangular-box-count} and \eqref{eqn:singular-curves-count} and \eqref{eqn:c-alpha-beta-constant-definition-proof-section} yields
\begin{align}\label{eqn:widetilde-global-count-proof-equation}
\# \widetilde{\mathcal{E}}(X; H_{\alpha, \beta}) = \left( 2 \left\lfloor \frac{X^{1/3}}{\alpha^{1/3}} \right\rfloor + 1 \right) \left( 2 \left\lfloor \frac{X^{1/2}}{\beta^{1/2}} \right\rfloor + 1 \right) - 2 \left\lfloor c(\alpha, \beta) X^{1/6} \right\rfloor - 1.
\end{align}

Now we proceed to prove formula \eqref{eqn:E-global-formula-generalized}. Starting from the decomposition \eqref{eqn:curve-families-decomposition} we obtain the formula
\begin{align*}
\# \widetilde{\mathcal{E}}(X; H_{\alpha, \beta}) = \sum_{1 \leq d \leq X^{1/12}} \# \left( d * \mathcal{E}(d^{-12}X; H_{\alpha, \beta}) \right) = \sum_{1 \leq d \leq X^{1/12}} \# \mathcal{E}(d^{-12}X; H_{\alpha, \beta}).
\end{align*}
Then, applying the Möbius inversion formula from Proposition \ref{prop:moebius-inversion-formula} to the functions $G(x) := \# \widetilde{\mathcal{E}}(x^{12};H_{\alpha,\beta})$ and $F(x) := \# \mathcal{E}(x^{12};H_{\alpha,\beta})$ with $x := X^{1/12}$, and using formula \eqref{eqn:widetilde-global-count-proof-equation} we obtain
\begin{align}\label{eqn:E-mu-equality}
 \# \mathcal{E}(X; H_{\alpha, \beta}) &= \sum_{1 \leq d \leq X^{1/12}} \mu(d)  \# \widetilde{\mathcal{E}}(d^{-12}X; H_{\alpha, \beta}) \notag\\
 &= \sum_{1 \leq d \leq X^{1/12}} \mu(d)\left(  \left( 2 \left\lfloor \frac{X^{1/3}}{d^4 \alpha^{1/3}} \right\rfloor + 1 \right) \left( 2 \left\lfloor \frac{X^{1/2}}{d^6 \beta^{1/2}} \right\rfloor + 1 \right) - 2 \left\lfloor c(\alpha, \beta) \frac{X^{1/6}}{d^2} \right\rfloor - 1  \right) \notag \\
 &= \sum_{1 \leq d \leq X^{1/12}} \mu(d) \left(  \left(\frac{2 X^{1/3}}{d^4 \alpha^{1/3}} + O(1)  \right) \left(\frac{2 X^{1/2}}{d^6 \beta^{1/2}} + O(1)  \right) - \frac{2 c(\alpha, \beta)}{d^2} X^{1/6} + O(1) \right) \notag\\
 &= \sum_{1 \leq d \leq X^{1/12}} \mu(d) \left( \frac{4}{d^{10} \alpha^{1/3} \beta^{1/2}} X^{5/6} + \frac{1}{d^6}O(X^{1/2}) \right) \notag\\
 &= \frac{4}{\alpha^{1/3} \beta^{1/2}}\left( \sum_{1 \leq d \leq X^{1/12}} \frac{\mu(d)}{d^{10}} \right) X^{5/6} + \left( \sum_{1 \leq d \leq X^{1/12}} \frac{\mu(d)}{d^6} \right) O(X^{1/2}),
\end{align}
where we have used that $\lfloor X\rfloor=X+O(1)$ and also in the next to last equality we only kept the largest of the big-O terms. Now, using Lemma \ref{lemma:moebius-sum-lemma} for the sums in equation \eqref{eqn:E-mu-equality}, we get that
\begin{align*}
\# \mathcal{E}(X; H_{\alpha, \beta}) &= \frac{4}{\alpha^{1/3} \beta^{1/2}} \left( \frac{1}{\zeta(10)} + O(X^{-9/12}) \right) X^{5/6} + \left( \frac{1}{\zeta(6)} + O(X^{-5/12}) \right) O(X^{1/2}) \\
&= \frac{4}{\alpha^{1/3} \beta^{1/2} \zeta(10)} X^{5/6}+O( X^{1/2}). 
\end{align*}
\end{proof}

Now we give a proof of Theorem \ref{thm:parametrization-and-exact-count-curves-with-j-introduction}, which we restate here for the convenience of the reader.

\begin{theorem}\label{theo: elliptic-j-invariant-exact-count}
Let $j \in \mathbb{Q} \smallsetminus \{0,1728\}$. Then the set of all elliptic curves $E_{A, B}: y^2=x^3+Ax+B$ with $A, B \in \mathbb{Z}$ that have $j\left(E_{A, B}\right)=j$ and $H_{\alpha, \beta}\left(E_{A, B}\right) \leq X$ can be explicitly described parametrically by
\begin{align}\label{eq: elliptic-j-inv-parametrization}
\widetilde{\mathcal{E}}_j\left(X ; H_{\alpha, \beta}\right) = \left\{ E_{A(j, m), B(j, m)} \suchthat \text {$m \in \mathbb{Z} \smallsetminus \{  0\}$ and $|m| \leq c(j; H_{\alpha, \beta}) X^{1/6}$} \right\}
\end{align}
where $A(j, m)$ and $B(j, m)$ are the integers defined by
\begin{align}
A(j,m)=\frac{N_{a(j)}^2 m^2}{a(j)}\quad \text{and} \quad  B(j,m)=\frac{N_{a(j)}^3 m^3}{a(j)}
\end{align}
and the constant $c(j; H_{\alpha, \beta})$ is defined by
\begin{align}
c(j; H_{\alpha, \beta}) := \min \left\{\frac{|a(j)|^{1 / 2}}{N_{a(j)} \alpha^{1 / 6}}, \frac{|a(j)|^{1 / 3}}{N_{a(j)} \beta^{1 / 6}}\right\}.
\end{align}
This implies that
\begin{align}\label{eqn:exact-j-count-generalized-height-formula}
\# \widetilde{\mathcal{E}}_j\left(X ; H_{\alpha, \beta}\right)=2\left \lfloor c(j; H_{\alpha, \beta}) X^{1 / 6}\right \rfloor.
\end{align}
Moreover, for any $m \in \Z \smallsetminus \{ 0 \}$ and any prime $p \in \mathbb{P}$ we have that 
\[
2\ord_p(m) \leq \ord_p{(A(j, m))} \leq 2\ord_p(m)+1 \quad \text{if $\ord_p{(a(j))} \geq 0$}
\]  
and 
\[
3\ord_p(m) \leq \ord_p{(B(j, m))} \leq 3\ord_p(m)+2 \quad \text{if $\ord_p{(a(j))} < 0$}.
\] 
In particular, this implies that the curve $E_{A(j,m),B(j,m)}\in\mathcal{E}_j(X;H_{\alpha,\beta})$ if and only if $m$ is square-free and $|m| \leq c(j; H_{\alpha, \beta}) X^{1/6}$, in other words, we have
\begin{align}\label{eqn:Q-isomorphism-classes-elliptic-j-inv-parametrization}
\mathcal{E}_j\left(X ; H_{\alpha, \beta}\right) = \left\{ E_{A(j, m), B(j, m)} \suchthat \text {$m \in \mathbb{Z} \smallsetminus \{  0\}$ is square-free and $|m| \leq c(j; H_{\alpha, \beta}) X^{1/6}$} \right\}.
\end{align}
\end{theorem}

\begin{proof}
Again, we note that $H_{\alpha,\beta}(E_{A, B}) \leq X$ if and only if $|A| \leq \dfrac{X^{1/3}}{\alpha^{1/3}}$ and $|B| \leq \dfrac{X^{1/2}}{\beta^{1/2}}$. Hence, the bound on the naive height $H_{\alpha,\beta}(E_{A, B}) \leq X$ is equivalent to the lattice point $(A, B) \in \Z^2$ lying in a rectangular box centered at the origin.
%, as seen in Figure \ref{fig:box}.
\begin{figure}[H]
\begin{tikzpicture}[scale=0.3]
% Draw axes
\draw[thick, ->] (-13,0) -- (13,0) node[right] {$x$};
\draw[thick, ->] (0,-17) -- (0,17) node[above] {$y$};
% Draw rectangle
\draw[thick] (-12.05, -16.10) rectangle (12.05, 16.10);
% Plot implicit curves
%Plot of the singular curves
\draw[very thick, GaloisRed, domain=-12:-0.005, samples=50] plot (\x, {sqrt(-4/27*\x^3)});
\draw[very thick, GaloisRed, domain=-12:-0.005, samples=50] plot (\x, {-sqrt(-4/27*\x^3)});
%Using j=-287496
\draw[very thick, blue, domain=-12.2:-0.005, samples=50] plot (\x, {sqrt(-0.041666*(\x)^3)});
\draw[very thick, blue, domain=-12.2:-0.005, samples=50] plot (\x, {-sqrt(-0.041666*(\x)^3)});

\foreach \x in {-12,-11,-10,...,12}
    \foreach \y in {-16,-15,-14,...,16}
        \fill[GaloisBlue] (\x,\y) circle           (0.2);

\foreach \m in {-2,-1,...,2}
    \fill[black] ({-3*(\m)^2}, {2*(\m)^3}) circle (0.3);
\fill[black](-6,3) circle (0.3);
\fill[black](-6,-3) circle (0.3);
\end{tikzpicture}
\caption{The box defined by the condition \( H^{\mathrm{cal}}(E_{A, B}) \leq X \) for \( X = 7000 \), with the lattice points \( (A, B) \in \mathbb{Z}^2 \) contained within it. The curves shown are the cuspidal cubic \( \mathcal{C}_{\frac{-4}{27}} \), associated with the singular curves (the outer red curve), and the cuspidal cubic \( \mathcal{C}_{a(j)} \) for \( j = -287496 \) (the inner blue curve). The integral points on these cuspidal cubics are highlighted with slightly larger circles.}
\label{fig:box}
\end{figure}
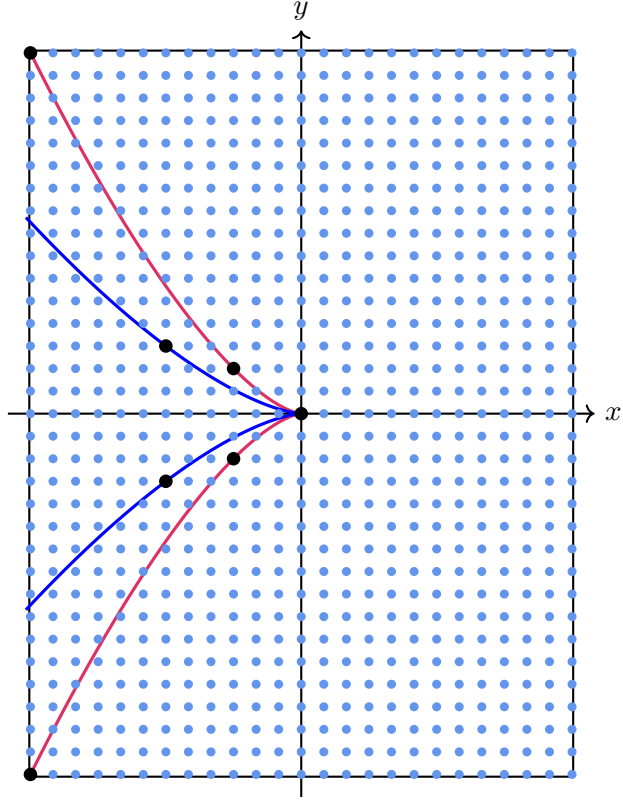
Next, Lemma \ref{lemma:no-singular-elliptic-curve-j-inv} implies that the quantity $\# \widetilde{\mathcal{E}}_j(X; H_{\alpha,\beta})$ is equal to the number of lattice points $(A, B) \in \Z^2 \smallsetminus{\{ (0, 0) \}}$ that lie on the cuspidal cubic $\mathcal{C}_{a(j)}: y^2=a(j)x^3$, with $a(j)=\frac{4(1728-j)}{27j}$ and that  $|A| \leq \dfrac{X^{1/3}}{\alpha^{1/3}}$ and $|B| \leq \dfrac{X^{1/2}}{\beta^{1/2}}$. Thus, in the notation of Proposition \ref{prop:cuspidal-cubic-exact-count}, we have a one-to-one correspondence
\begin{align*}
 \widetilde{\mathcal{E}}_j(X;H_{\alpha,\beta}) &\longleftrightarrow  \mathcal{C}_{a(j)}\left(\mathbb{Z};\frac{X^{1/3}}{\alpha^{1/3}}, \frac{X^{1/2}}{\beta^{1/2}}\right)\smallsetminus \{(0,0)\}.\\
 E_{A, B} &\longleftrightarrow (A, B)
\end{align*}
Then Proposition \ref{prop:cuspidal-cubic-exact-count} with the choices $a=a(j)$, $T_1=\dfrac{X^{1/3}}{\alpha^{1/3}}$ and $T_2=\dfrac{X^{1/2}}{\beta^{1/2}}$, yields the parametrization 
\begin{align*}
\widetilde{\mathcal{E}}_j\left(X ; H_{\alpha, \beta}\right)=\left\{ E_{A(j, m), B(j, m)} \suchthat \text {$m \in \mathbb{Z} \smallsetminus \{  0\}$ and $|m| \leq c(j;H_{\alpha,\beta})X^{1/6}$} \right\},
\end{align*}
where $c(j;H_{\alpha,\beta})= \min{\left\{\dfrac{|a|^{1/2}}{N_a\alpha^{1/6}}, \dfrac{|a|^{1/3}}{N_a\beta^{1/6}}\right\}}$ and where $ A(j, m)$ and $B(j, m)$ are the integers defined by
\begin{align*}
A(j,m)=\frac{N_{a(j)}^2 m^2}{a(j)}\quad \text{and} \quad  B(j,m)=\frac{N_{a(j)}^3 m^3}{a(j)}.
\end{align*}
This implies that
\begin{align*}
\# \widetilde{\mathcal{E}}_j(X;H_{\alpha,\beta})=2\left\lfloor\min{\left\{\frac{|a(j)|^{1/2}}{N_{a(j)}\alpha^{1/6}},\frac{|a(j)|^{1/3}}{N_{a(j)}\beta^{1/6}}\right\}}\cdot X^{1/6}\right\rfloor=2\left\lfloor c(j; H_{\alpha,\beta}) X^{1/6}\right\rfloor.
\end{align*}
We now prove the last inequalities stated in the theorem.
Recall that for a rational number $a \in \Q^{\times}$, we have $\ord_p{(N_a)} := \alpha_p(a)$, where
\[
\alpha_{p}(a):= \begin{cases} 
\left \lceil\dfrac{\operatorname{ord}_{p}(a)}{2} \right \rceil & \text{if $\operatorname{ord}_{p}(a) \geq 0$,} \vspace{5pt}\\ \left \lceil\dfrac{\operatorname{ord}_{p}(a)}{3} \right \rceil & \text{if $\operatorname{ord}_{p}(a) < 0$.}
\end{cases}
\]
In particular, when $\ord_p{(a)} \geq 0$, we have
\begin{align}\label{eqn: alpha_p-inequalities-1}
\frac{\ord_p{(a)}}{2} \leq \alpha_p{(a)} < \frac{\ord_p{(a)}}{2} + 1
\end{align}
On the other hand, when $\ord_p(a) < 0$, we have
\begin{align}\label{eqn:alpha_p-inequalities-2}
\frac{\ord_p{(a)}}{3} \leq \alpha_p{(a)} < \frac{\ord_p{(a)}}{3} + 1.
\end{align}
Let $m\in\Z\smallsetminus\{0\}$ and $p\in\mathbb{P}$. If $\ord_p(a(j))\geq 0$, then we have that
\begin{align*}
\ord_p(A(j,m))=\ord_p\left(\frac{N^2_{a(j)}m^2}{a(j)}\right) &= 2\ord_p(N_{a(j)})+2\ord_p(m)-\ord_p(a(j))\\
&= 2\left\lceil\frac{\ord_p(a(j))}{2}\right\rceil+2\ord_p(m)-\ord_p(a(j)).
\end{align*}
Then, using \eqref{eqn: alpha_p-inequalities-1} yields
\begin{align*}
2\ord_p(m) \leq \ord_p{(A(j, m))} \leq 2\ord_p(m)+1.
\end{align*}
Similarly, if $\ord_p(a(j)) < 0$, then 
\begin{align*}
\ord_p(B(j,m))=\ord_p\left(\frac{N^3_{a(j)}m^3}{a(j)}\right) &= 3\ord_p(N_{a(j)})+3\ord_p(m)-\ord_p(a(j))\\
&= 3\left\lceil\frac{\ord_p(a(j))}{3}\right\rceil+3\ord_p(m)-\ord_p(a(j))
\end{align*}
and using \eqref{eqn:alpha_p-inequalities-2}, we conclude that
\begin{align*}
3\ord_p(m) \leq \ord_p{(B(j, m))} \leq 3\ord_p(m)+2 \quad \text{if $\ord_p{(a(j))} < 0$}.
\end{align*}
Finally, these inequalities imply that $E_{A(j,m),B(j,m)}\in\mathcal{E}_j(X;H_{\alpha,\beta})$ if and only if $m$ is square-free.
\end{proof}

\begin{remark}\label{remark:construction-Q-isomorphism-representative}
The parametrization in \eqref{eq: elliptic-j-inv-parametrization} not only enables the enumeration of curves in $ \widetilde{\mathcal{E}}_j $ by generalized naive height, but also identifies the $ \Q $-isomorphism class representative of any such curve. Given $ E_{A', B'} \in \widetilde{\mathcal{E}}_j $, we recover its representative $ E_{A, B} $ by setting $ A = A'/d^4 $ and $ B = B'/d^6 $, where $ d $ is the integer defined in Lemma \ref{lemma: elliptic-family-decomposition}. Since $ E_{A, B} \in \mathcal{E}_j $, this curve is the unique representative of the $ \Q $-isomorphism class. Moreover, a curve $ E_{A, B} \in \widetilde{\mathcal{E}}_j $ is a representative if and only if the corresponding parameter $ m $ in \eqref{eq: elliptic-j-inv-parametrization} is square-free.

\end{remark}

Recall that an integer $n$ is called $k$-free if there is no prime $p\in\Z$ such that $p^k|n$. The following proposition states a well-known asymptotic formula for the function $Q_k(X)$, which counts the number of $k$-free integers in the interval $[1,X]$. A proof in general can be found in \cite[Theorem 1.1]{EL29} (for a more recent reference see also \cite[Proposition 2.2]{MV07} for the $k = 2$ case and \cite[Exercise 2.1.1 (3)]{MV07} for the general case). This result will be used in the proofs of our asymptotic formulas for counting elliptic curves.

\begin{proposition}\label{prop: k-free-numbers-asymptotics}
For $X\in\R_{\geq 0}$ and $k\geq 2$, let $Q_k(X)$ denote the number of $k$-free integers $n$ such that $1 \leq n\leq X$. Then
\begin{align}
Q_k(X)=\frac{1}{\zeta(k)}X+O(X^{1/k})
\end{align}
as $X\rightarrow\infty$.
\end{proposition}

We now give the proof of Theorem \ref{thm:fixed-j-theorem-generalized}, which we state again for the convenience of the reader.

\begin{theorem}\label{theo: elliptic-curve-formulas}
Let $j \in \Q$. Then for every $X > 1$, the number of elliptic curves in the families $\mathcal{E}_j(X;H_{\alpha,\beta})$ and $\widetilde{\mathcal{E}}_j(X;H_{\alpha,\beta})$ satisfies the following asymptotic formulas.
For $j=0$, we have
\begin{align}
\# \mathcal{E}_0(X;H_{\alpha,\beta}) = \frac{2}{\beta^{1/2}\zeta(6)}X^{1/2} + O(X^{1/12}),
\end{align}
and
\begin{align}
\#\widetilde{\mathcal{E}}_0(X;H_{\alpha,\beta}) = 2\left\lfloor\frac{X^{1/2}}{\beta^{1/2}}\right\rfloor = \frac{2}{\beta^{1/2}}X^{1/2}+O(1).
\end{align}
For $j=1728$, we have
\begin{align}
\# \mathcal{E}_{1728}(X;H_{\alpha,\beta}) = \frac{2}{\alpha^{1/3}\zeta(4)}X^{1/3} + O(X^{1/12})
\end{align}
and
\begin{align}
\#\widetilde{\mathcal{E}}_{1728}(X;H_{\alpha,\beta}) = 2\left\lfloor\frac{X^{1/3}}{\alpha^{1/3}}\right\rfloor = \frac{2}{\alpha^{1/3}}X^{1/3}+O(1).
\end{align}
Finally, for $j\neq 0,1728$, we have
\begin{align}\label{eqn:elliptic-count-j-asymptotic}
\# \mathcal{E}_j(X;H_{\alpha,\beta}) = \frac{2}{\zeta(2)}c(j; H_{\alpha,\beta}) X^{1/6} + O(X^{1/12})
\end{align}
and
\begin{align}\label{eqn:elliptic-count-j-exact}
\# \widetilde{\mathcal{E}}_j(X;H_{\alpha,\beta}) = 2\left\lfloor c(j;H_{\alpha,\beta}) X^{1/6}\right\rfloor
\end{align}
where $c(j; H_{\alpha,\beta})$ is the constant defined in Theorem \ref{theo: elliptic-j-invariant-exact-count}.
\end{theorem}

\begin{proof}

We consider separately the three cases where the \( j \)-invariant is equal to 0, equal to 1728, or lies in \( \mathbb{Q} \smallsetminus \{0, 1728\} \).

\noindent \textbf{Curves with $j = 0$}

If a curve $E_{A,B}$ has $j$-invariant $0$, then $A=0$. Therefore, the set $\widetilde{\mathcal{E}}_0(X;H_{\alpha,\beta})$ corresponds to the points $(0,B)\in\Z^2$ with $ B\neq 0$ and $|B|\leq \dfrac{X^{1/2}}{\beta^{1/2}}$. Thus, since $\lfloor X \rfloor = X + O(1)$, we have
\[
\#\widetilde{\mathcal{E}}_0(X;H_{\alpha,\beta})=2\left\lfloor\frac{X^{1/2}}{\beta^{1/2}}\right\rfloor=\frac{2}{\beta^{1/2}}X^{1/2}+O(1).
\]
Furthermore, the set $\mathcal{E}_0(X;H_{\alpha,\beta})$ corresponds to points $(0,B)\in \Z^2$ such that $B\neq 0$,  $|B|\leq \dfrac{X^{1/2}}{\beta^{1/2}}$ and $B$ is $6$-free. Thus, by Proposition \ref{prop: k-free-numbers-asymptotics}, we have
\[
\# \mathcal{E}_0(X;H_{\alpha,\beta})=\frac{2}{\beta^{1/2}\zeta(6)}X^{1/2}+O(X^{1/12}).
\]

\noindent \textbf{Curves with $j = 1728$}

On the other hand, if a curve $E_{A,B}$ has $j$-invariant $1728$, then $B=0$. Thus, the set $\widetilde{\mathcal{E}}_{1728}(X;H_{\alpha,\beta})$ corresponds to points $(A,0)\in\Z^2$ with $A\neq 0$ and $|A|\leq\dfrac{X^{1/3}}{\alpha^{1/3}}$. Then, we have
\[
\#\widetilde{\mathcal{E}}_{1728}(X;H_{\alpha,\beta})=2\left\lfloor\frac{X^{1/3}}{\alpha^{1/3}}\right\rfloor=\frac{2}{\alpha^{1/3}}X^{1/3}+O(1).
\]
Moreover, the set $\mathcal{E}_{1728}(X;H_{\alpha,\beta})$ corresponds  with points $(A,0)\in\Z^2$ such that $A\neq 0$, $|A|\leq \dfrac{X^{1/3}}{\alpha^{1/3}}$ and $A$ is 4-free. Therefore, in this case Proposition \ref{prop: k-free-numbers-asymptotics} yields
\[
\# \mathcal{E}_{1728}(X;H_{\alpha,\beta})=\frac{2}{\alpha^{1/3}\zeta(4)}X^{1/3}+O(X^{1/12}).
\]

\noindent \textbf{Curves with $j \in \Q \smallsetminus \{ 0, 1728 \}$}

Now, let $j\in\Q\smallsetminus\{0,1728\}$. In Theorem \ref{theo: elliptic-j-invariant-exact-count} we already proved that
\[
\# \widetilde{\mathcal{E}}_j(X;H_{\alpha,\beta})=2\left\lfloor c(j;H_{\alpha,\beta}) X^{1/6}\right\rfloor
\]
and moreover, by the parametrization \eqref{eqn:Q-isomorphism-classes-elliptic-j-inv-parametrization}, that the set $\mathcal{E}_j(X;H_{\alpha,\beta})$ is in one to one correspondence with the set of $m \in \Z \smallsetminus \{ 0 \}$ such that $m$ is square-free and $|m|\leq c(j;H_{\alpha,\beta})X^{1/6}$. Then, again using Proposition \ref{prop: k-free-numbers-asymptotics}, we conclude that
\begin{align*}
\# \mathcal{E}_j(X;H_{\alpha,\beta}) = \frac{2}{\zeta(2)}c(j; H_{\alpha,\beta}) X^{1/6}+O(X^{1/12}).
\end{align*}

\end{proof}

Finally, as a direct consequence of Theorem \ref{theo: elliptic-curve-formulas} we prove Theorem \ref{thm:CM-theorem-generalized}, which we restate for the convenience of the reader.

\begin{theorem}\label{theo: CM-elliptic-count}
For every $X > 1$, the numbers of CM elliptic curves in the families $\mathcal{E}^{\mathrm{cm}}(X;H_{\alpha,\beta})$ and $\widetilde{\mathcal{E}}^{\mathrm{cm}}(X;H_{\alpha,\beta})$ satisfy the formulas
\begin{align}\label{eq:Ecm-asymptotics}
\#\mathcal{E}^{\mathrm{cm}}(X;H_{\alpha,\beta}) = \frac{2}{\beta^{1/2} \zeta(6)} X^{1/2} + \frac{2}{ \alpha^{1/3}\zeta(4)} X^{1/3} +\frac{2}{\zeta(2)}K(\alpha,\beta)X^{1/6} + O(X^{1/12})
\end{align}
and
\begin{align}\label{eq:Ecm-hat-asymptotics}
\#\widetilde{\mathcal{E}}^{\mathrm{cm}}(X;H_{\alpha,\beta}) = 2\left\lfloor\frac{X^{1/2}}{\beta^{1/2}}\right\rfloor+ 2\left\lfloor\frac{X^{1/3}}{\alpha^{1/3}}\right\rfloor  + \sum_{\substack{j \in \mathcal{J}^{\mathrm{cm}}\\ j \neq 0, 1728}} 2\left\lfloor c(j;H_{\alpha,\beta})X^{1/6} \right\rfloor,  
\end{align}
where
\begin{align}
K(\alpha,\beta) := \sum_{\substack{j \in\mathcal{J}^{\mathrm{cm}} \\ j \neq 0, 1728}}c(j;H_{\alpha,\beta}).
\end{align}
\end{theorem}

\begin{proof}
By the decomposition \eqref{eq:Ecm-decomposition} we have
\begin{align*}
    \widetilde{\mathcal{E}}^{\mathrm{cm}}(X; H_{\alpha, \beta}) = \coprod_{j \in \mathcal{J}^{\textrm{cm}}} \widetilde{\mathcal{E}}_j(X; H_{\alpha, \beta}) \quad \text{and} \quad \mathcal{E}^{\mathrm{cm}}(X; H_{\alpha, \beta}) = \coprod_{j \in \mathcal{J}^{\textrm{cm}}} \mathcal{E}_j (X; H_{\alpha, \beta}).
\end{align*}
This implies that
\begin{align*}
    \# \widetilde{\mathcal{E}}^{\mathrm{cm}}(X; H_{\alpha, \beta}) = \sum_{j \in \mathcal{J}^{\textrm{cm}}} \# \widetilde{\mathcal{E}}_j(X; H_{\alpha, \beta}) \quad \text{and} \quad \# \mathcal{E}^{\mathrm{cm}}(X; H_{\alpha, \beta}) = \sum_{j \in \mathcal{J}^{\textrm{cm}}} \# \mathcal{E}_j (X; H_{\alpha, \beta}).
\end{align*}
Then the theorem follows by plugging in the formulas from Theorem \ref{theo: elliptic-curve-formulas}.
\end{proof}

\subsection{Proofs of Theorems \ref{thm:distribution-reps-all-curves} and \ref{thm:distribution-reps-fixed-j}} Theorems \ref{thm:distribution-reps-all-curves} and \ref{thm:distribution-reps-fixed-j} follow from an application of the following lemma to the asymptotic formulas given in Theorems \ref{thm:global-counts-generalized} and \ref{theo: elliptic-curve-formulas}, respectively.

\begin{lemma}
Let $a, b, c \in \R$ with $a>0$ and $b,c<a$. Suppose that $A(X)$ and $B(X)$ are real-valued functions that satisfy the asymptotic formulas
\[
A(X)=C X^{a}+O(X^{b})
\quad \text{and} \quad
B(X)=D X^{a}+O(X^{c}),
\]
for some constants $C,D$ with $D\neq 0$. Then, as $X\to\infty$, we have
\[
\frac{A(X)}{B(X)} =
\frac{C}{D} + O\!\left(X^{-(a-\max\{b,c\})}\right).
\]
\end{lemma}

\begin{proof}
Factoring the main terms we get
\[
\frac{A(X)}{B(X)}
= \frac{C X^{a}\left(1+O\!\left(X^{b-a}\right)\right)}{D X^{a}\!\left(1+O\left(X^{c-a}\right)\right)} = 
\frac{C}{D}
\,
\frac{1+O\!\left(X^{b-a}\right)}{1+O\!\left(X^{c-a}\right)}.
\]
Now, since $c<a$, we have that $X^{c-a}\to 0$, and therefore
\[
\frac{1}{1+O\!\left(X^{c-a}\right)}
=
1+O\!\left(X^{c-a}\right).
\]
It follows that
\[
\frac{A(X)}{B(X)}
=
\frac{C}{D}
\Big(1+O\!\Big(X^{b-a}\Big)\Big)
\left(1+O\!\left(X^{c-a}\right)\right)
=
\frac{C}{D}
+
O\!\left(X^{-(a-\max\{b,c\})}\right).
\]
\end{proof}

\section{Explicit computations and data for CM elliptic curves}\label{section:explicit-computations}
In this section, we present explicit computations related to CM elliptic curves ordered by naive height. We include numerical data for very large values of the height, extending beyond the ranges typically considered in experimental computations by taking advantage of our theoretical results. In particular, Table \ref{table: CM-elliptic-minimal-naive-height} gives short Weierstrass equations for the curves of minimal naive height associated to each CM $j$-invariant, while Table \ref{table:Predicted-elliptic-curves-for-given-height} provides counts of CM elliptic curves corresponding to each of the thirteen possible CM orders with class number one, up to height $10^{30}$. These computations offer additional insight into the distribution and structure of CM curves in practice and serve as a computational complement to the theoretical developments discussed earlier.

\begin{table}[H]
\centering
{\tabulinesep=1.2mm
\begin{tabu}{|c|c|c|c|c|}
\hline
$d_K$ & $f$ & CM $j$-invariant & Curves with minimal height & Minimal naive height $H^{\mathrm{cal}}$ \\
\hline
\multirow{3}{*}{$-3$} 
  & 1 & $0$ & $y^2 = x^3 \pm 1$ & 27 \\
%\cline{2-5}
  & 2 & $2^4 \cdot 3^3 \cdot 5^3$ & $y^2 = x^3 - 15x \pm 22$ & \num{13500} \\
%\cline{2-5}
  & 3 & $-2^{15} \cdot 3 \cdot 5^3$ & $y^2 = x^3 - 120x \pm 506$ & \num{6912972} \\
\hline
\multirow{2}{*}{$-4$} 
  & 1 & $2^6 \cdot 3^3 = 1728$ & $y^2 = x^3 + x$ & 4 \\
%\cline{2-5}
  & 2 & $2^3 \cdot 3^3 \cdot 11^3$ & $y^2 = x^3 - 11x \pm 14$ & \num{5324} \\
\hline
\multirow{2}{*}{$-7$}
  & 1 & $-3^3 \cdot 5^3$ & $y^2 = x^3 - 35x \pm 98$ & \num{259308} \\
%\cline{2-5}
  & 2 & $3^3 \cdot 5^3 \cdot 17^3$ & $y^2 = x^3 - 595x \pm 5586$ & \num{842579500} \\
\hline
$-8$ & 1 & $2^6 \cdot 5^3$ & $y^2 = x^3 - 30x \pm 56$ & \num{108000} \\
\hline
$-11$ & 1 & $-2^{15}$ & $y^2 = x^3 - 264x \pm 1694$ & \num{77480172} \\
\hline
$-19$ & 1 & $-2^{15} \cdot 3^3$ & $y^2 = x^3 - 152x \pm 722$ & \num{14074668} \\
\hline
$-43$ & 1 & $-2^{18} \cdot 3^3 \cdot 5^3$ & $y^2 = x^3 - 3440x \pm 77658$ & \num{162830654028} \\
\hline
$-67$ & 1 & $-2^{15} \cdot 3^3 \cdot 5^3 \cdot 11^3$ & $y^2 = x^3 - 29480x \pm 1948226$ & \num{102480782771052} \\
\hline
$-163$ & 1 & $-2^{18} \cdot 3^3 \cdot 5^3 \cdot 23^3 \cdot 29^3$ & $y^2 = x^3 - 8697680x \pm 9873093538$ & \num{2631905352272628650988} \\
\hline
\end{tabu}}
\caption{The CM elliptic curves over $\Q$ of minimal calibrated naive height for each of the thirteen possibilities for the endomorphism ring.}
\label{table: CM-elliptic-minimal-naive-height}
\end{table}

\begin{table}[H]
\centering
{\tabulinesep=1.2mm
\begin{tabu}{|c|c|c|*{5}{c}|}
\cline{4-8}
\multicolumn{3}{c|}{} 
  & \multicolumn{5}{c|}{$\# \widetilde{\mathcal{E}}_j(X, H^{\mathrm{cal}})$ for $X = 10^n$} \\
%\cline{4-8} 
\hline
$d_K$ & $f$ & CM $j$-invariant 
 & $10^{10}$ & $10^{15}$ & $10^{20}$ & $10^{25}$ & $10^{30}$ \\
\hline
\multirow{3}{*}{$-3$} 
  & 1 & $0$ & 38490 & 12171612 & 3849001794 & 1217161238900 & 384900179459750 \\
  & 2 & $2^4 \cdot 3^3 \cdot 5^3$ & 18 & 128 & 882 & 6014 & 40986 \\
  & 3 & $-2^{15} \cdot 3 \cdot 5^3$ & 6 & 44 & 312 & 2126 & 14490 \\
\hline
\multirow{2}{*}{$-4$} 
  & 1 & $2^6 \cdot 3^3 = 1728$ & 2714 & 125992 & 5848034 & 271441760 & 12599210498 \\
  & 2 & $2^3 \cdot 3^3 \cdot 11^3$ & 22 & 150 & 1030 & 7024 & 47860 \\
\hline
\multirow{2}{*}{$-7$}
  & 1 & $-3^3 \cdot 5^3$ & 10 & 78 & 538 & 3676 & 25044 \\
  & 2 & $3^3 \cdot 5^3 \cdot 17^3$ & 2 & 20 & 140 & 954 & 6506 \\
\hline
$-8$ & 1 & $2^6 \cdot 5^3$ & 12 & 90 & 624 & 4252 & 28980 \\
\hline
$-11$ & 1 & $-2^{15}$ & 4 & 30 & 208 & 1420 & 9686 \\
\hline
$-19$ & 1 & $-2^{15} \cdot 3^3$ & 4 & 40 & 276 & 1888 & 12870 \\
\hline
$-43$ & 1 & $-2^{18} \cdot 3^3 \cdot 5^3$ & 0 & 8 & 58 & 396 & 2706 \\
\hline
$-67$ & 1 & $-2^{15} \cdot 3^3 \cdot 5^3 \cdot 11^3$ & 0 & 2 & 18 & 134 & 924 \\
\hline
$-163$ & 1 & $-2^{18} \cdot 3^3 \cdot 5^3 \cdot 23^3 \cdot 29^3$ & 0 & 0 & 0 & 6 & 52 \\
\hline
\multicolumn{3}{|c|}{$\# \widetilde{\mathcal{E}}^{\mathrm{cm}}(X, H^{\mathrm{cal}})$}
  & 41282 & 12418204 & 3932003556 & 1240837801110 & 392324168837610 \\
\hline
\end{tabu}}
\caption{Cardinalities $\# \widetilde{\mathcal{E}}_j(X, H^{\mathrm{cal}})$ for CM $j$-invariants at various height bounds $X$, computed with the exact formulas \eqref{eqn:exact-formulas-calibrated-naive-height}. The final row shows the total count over all CM $j$-invariants.}
\label{table:Predicted-elliptic-curves-for-given-height}
\end{table}

For discriminant $d_K=-163$, Table \ref{table:Predicted-elliptic-curves-for-given-height} shows that there are exactly six elliptic curves $E_{A,B}\in\widetilde{\mathcal{E}}_j(10^{25}, H^{\mathrm{cal}})$ with $j$-invariant $j=-2^{18} \cdot 3^3 \cdot 5^3 \cdot 23^3 \cdot 29^3$. These curves are the ones obtained by taking $m \in \{\pm 1, \pm 2, \pm 3 \}$ in the parametrized family of curves \eqref{eqn:curves-163-parametrization} given by
\begin{align*}
E_{A(m), B(m)} \colon y^2 = x^3  + A(m)x + B(m),
\end{align*}
where
\begin{align*}
A(m) := -(2^{4} \cdot 5 \cdot 23 \cdot 29 \cdot 163)m^2 \quad \text{and} \quad
B(m) := -(2 \cdot 7 \cdot 11 \cdot 19 \cdot 127 \cdot 163^{2})m^3.
\end{align*}
Explicitly, these curves are the following:
\begin{align*}
E_{A(\pm 1),B(\pm 1)} \colon &y^2=x^3 - 8697680x \pm 9873093538, \\
E_{A(\pm 2),B(\pm 2)} \colon &y^2=x^3-34790720x \pm 78984748304,\\
E_{A(\pm 3),B(\pm 3)} \colon &y^2=x^3-78279120x\pm 266573525526.
\end{align*}

In Theorem \ref{thm:fixed-j-theorem-generalized}, we showed that for every $j\in\Q\smallsetminus\{0,1728\}$, the number of elliptic curves in the family $\widetilde{\mathcal{E}}_j(X;H^{\mathrm{cal}})$ satisfies the asymptotic formula
\begin{align}\label{eqn:asymptotic-formula-fixed-j-computation-section}
\#\mathcal{E}_j(X; H^{\mathrm{cal}})=\dfrac{2c(j; H^{\mathrm{cal}})}{\zeta(2)} X^{1/6} + O(X^{1/12}).
\end{align}
Now, in Table~\ref{table: Calculated-terms-j-fixed}, we compute the value of the leading coefficient in this formula for each CM $j$-invariant $j \in \mathcal{J}^{\mathrm{cm}} \smallsetminus \{0, 1728\}$, using \texttt{SageMath}. The reader should observe in Table \ref{table: Calculated-terms-j-fixed} that the constant $\dfrac{2c(j;H^{\mathrm{cal}})}{\zeta(2)}$ corresponding to $d_k=-163$ and $j=-2^{18} \cdot 3^3 \cdot 5^3 \cdot 23^3 \cdot 29^3$ has the smallest value, a fact that is directly related to the observation made in Table  \ref{table: CM-elliptic-minimal-naive-height} that it is precisely for this $j$-invariant that the family $\widetilde{\mathcal{E}}_j(X;H^{\mathrm{cal}})$ has the largest minimal naive height for the elliptic curves with complex multiplication.

\begin{table}[H]
\centering
{\tabulinesep=1.2mm
\begin{tabu}{c c c l} \hline
$d_K$ & $f$ & CM $j$-invariant & $\dfrac{2c(j;H^{\mathrm{cal}})}{\zeta(2)}$ \\ \hline
$-3$ & $2$ & $2^{4} \cdot 3^{3} \cdot 5^{3}$ & $0.2491681566$ \\ 
$-3$ & $3$ & $- 2^{15} \cdot 3 \cdot 5^{3}$ & $0.08809218206$ \\ 
$-4$ & $2$ & $2^{3} \cdot 3^{3} \cdot 11^{3}$ & $0.2909657203$ \\ 
$-7$ & $1$ & $- 3^{3} \cdot 5^{3}$ & $0.1522575538$ \\ 
$-7$ & $2$ & $3^{3} \cdot 5^{3} \cdot 17^{3}$ & $0.03956213184$ \\ 
$-8$ & $1$ & $2^{6} \cdot 5^{3}$ & $0.1761884932$ \\ 
$-11$ & $1$ & $- 2^{15}$ & $0.05888658982$ \\ 
$-19$ & $1$ & $- 2^{15} \cdot 3^{3}$ & $0.07824834141$ \\ 
$-43$ & $1$ & $- 2^{18} \cdot 3^{3} \cdot 5^{3}$ & $0.01645351934$ \\ 
$-67$ & $1$ & $- 2^{15} \cdot 3^{3} \cdot 5^{3} \cdot 11^{3}$ & $0.005620493221$ \\ 
$-163$ & $1$ & $-2^{18} \cdot 3^{3} \cdot 5^{3} \cdot 23^{3} \cdot 29^{3}$ & $0.0003272174502$ \\ \hline
\end{tabu}}
\caption{Values of the main coefficients $\dfrac{2c(j;H^{\mathrm{cal}})}{\zeta(2)}$ in the asymptotic formulas \eqref{eqn:asymptotic-formula-fixed-j-computation-section} for the calibrated naive height, for the distinct CM $j$-invariants $j \in \mathcal{J}^{\mathrm{cm}} \smallsetminus \{ 0, 1728 \}$.}
\label{table: Calculated-terms-j-fixed}
\end{table}

\section{Acknowledgments}
The first author would like to thank Diana Ramírez Leitón for her helpful advice on preparing and documenting the code for publication on GitHub. The authors thank the Centro de Investigación en Matemática Pura y Aplicada (CIMPA) and the School of Mathematics of the University of Costa Rica for their administrative help and support during this project. This paper was written as part of research project 821-C5-139, led by the first author and registered with the Vicerrectoría de Investigación of the University of Costa Rica.

% \newpage
\bibliography{elliptic-counts-biblio.bib}
\bibliographystyle{alphaurl}

\end{document}